\documentclass[a4paper]{amsart}
\usepackage[margin=1in]{geometry}
\usepackage{verbatim}
\usepackage{multicol}

\usepackage{tikz}
\usetikzlibrary{calc}
\usetikzlibrary{arrows}
\usetikzlibrary{decorations.markings}
\usetikzlibrary{decorations.pathreplacing}
\usetikzlibrary{arrows,shapes,positioning}%,snakes}
\usetikzlibrary{shapes.geometric}
\usetikzlibrary{patterns}
\tikzstyle directed=[postaction={decorate,decoration={markings,
    mark=at position #1 with {\arrow{>}}}}]
\tikzstyle rdirected=[postaction={decorate,decoration={markings,
    mark=at position #1 with {\arrow{<}}}}]
\tikzset{anchorbase/.style={baseline={([yshift=-0.5ex]current bounding box.center)}}}

\tikzset{
    partial ellipse/.style args={#1:#2:#3}{
        insert path={+ (#1:#3) arc (#1:#2:#3)}
    }
  }

\usepackage{tikz-cd}

\usepackage{todonotes}

\usepackage{hyperref}
\usepackage[capitalize]{cleveref}
\crefname{figure}{Figure}{Figures}

\newcommand{\R}{\mathbb{R}}
\newcommand{\C}{\mathbb{C}}
\newcommand{\Z}{\mathbb{Z}}

\newcommand{\sll}{\mathfrak{sl}}

\newcommand{\zz}[1]{\mathbf{#1}}

\DeclareMathOperator{\Hom}{Hom}

\DeclareMathOperator{\Id}{Id}

\let\tilde=\widetilde

\newcommand{\calC}{\mathcal{C}}
\newcommand{\calP}{\mathcal{P}}
\newcommand{\id}{\mathrm{id}}

\renewcommand{\tilde}{\widetilde}
\theoremstyle{plain}
\newtheorem{theorem}{Theorem}[section]
\newtheorem{lemma}[theorem]{Lemma}
\newtheorem{corollary}[theorem]{Corollary}
\newtheorem{remark}[theorem]{Remark}
\newtheorem{definition}[theorem]{Definition}

\newtheorem{proposition}[theorem]{Proposition}
\newtheorem{question}[theorem]{Question}

\title[Extriangulated structures and braids]{Extriangulated structures from stability conditions and application to braid representations}

\author{Hoel Queffelec}
\address{France-Australia Mathematical Sciences and Interactions ANU - CNRS International Research Laboratory \\ Australia}
\email{hoel.queffelec@cnrs.fr}

\author{Anne-Laure Thiel}
\address{Université Bourgogne Europe, CNRS, IMB UMR 5584, 21000 Dijon, France}
\email{Anne-Laure.Thiel@ube.fr} 

\author{Emmanuel Wagner}
\address{Univ Paris Cité, IMJ-PRG, Univ Paris Sorbonne, UMR 7586 CNRS, F-75013, Paris, France}
\email{wagner@imj-prg.fr}

\date{}

\usepackage{fancyhdr}

\begin{document}

\begin{abstract}
We use the notion of Bridgeland stability condition and its associated metric to endow triangulated categories with extriangulated structures and study their extriangulated Grothendieck groups. This study is motivated by Khovanov-Seidel's categorification of the Burau representation, from which can be extracted the Lawrence-Krammer-Bigelow representation, providing a relation between these two faithful representations of the braid groups.    
\end{abstract}

\maketitle

\section{Introduction}

The Burau representation of the braid group was introduced by Burau~\cite{Burau} in 1934, who gave explicit formulas for it. It has since been reinterpreted through different lenses. The one closest to our considerations goes as follows. Consider the $n$-strand braid group $B_{n}$, and $V$ the vector representation of $U_q(\sll_{2})$. The braid group then acts on $V^{\otimes n}$ in a deformed version of Schur-Weyl duality, using $R$-matrices, or equivalently factoring through the Hecke algebra. The space $V^{\otimes n}$ can be decomposed into weight spaces indexed by $n,n-2,\dots,-n$, and because the action of $B_n$ commutes with the action of $U_q(\sll_2)$, it follows that each of these weight spaces is a representation of $B_n$. The one in weight $n$ is the trivial representation, and the one in weight $n-2$ is the Burau representation.

The question of the faithfulness of the Burau representation has received a lot of attention. It has been known for a long time that in the $n=3$ case, the representation is faithful~\cite{MagnusPeluso}. Moody, Long-Paton and Bigelow proved in several steps that the representation is not faithful whenever $n\geq 5$~\cite{Moody,LongPaton,Bigelow_Burau}. This leaves the case $n=4$ open.

This faithfulness defect was interestingly solved in two rather different ways at about the same time: Krammer and Bigelow, building on the work of Lawrence, provided a more complicated faithful finite-dimensional representation~\cite{Lawrence,Krammer,Bigelow_linearity}. On the other hand, Khovanov and Seidel upgraded the Burau representation to an action on a category and proved its faithfulness~\cite{KhS}. The goal of our paper is to relate both representations.

\subsection{LKB representation}

The LKB representation originates from homological considerations related to configuration spaces~\cite{Lawrence,Bigelow_linearity}, but can also be given a more algebraic taste~\cite{Krammer}. Closer to the presentation of the Burau representation given above, it can appear in the $n$-fold tensor product of a Verma module for $\sll_2$~\cite{Kohno,Martel,LTV_LKB} by focusing on the weight space in top weight minus $4$. Explicit formulas appear in different places; the ones we present in Section~\ref{sec:LKBdef} come from~\cite{PaPa}.

Faithfulness results for the LKB representation were proven by Krammer and Bigelow. Bigelow used a topological approach, while Krammer had an inductive proof using Garside structure, but both versions are similar in that they prove that some coefficients in top degree won't vanish.

\subsection{Khovanov-Seidel representation}

Khovanov and Seidel's approach to solving the faithfulness defect for the Burau representation took a different path~\cite{KhS}. Instead of increasing the dimension of the representation, but keeping it linear, they replaced it with a category, with generating objects corresponding to the basis of the Burau representation. We review the construction of the category in Section~\ref{sec:KhS}. For the moment, let us simply say that Khovanov and Seidel define a triangulated category, with an explicit list of indecomposable projectives that are used to build complexes up to homotopy. Braids then act by autoequivalences of the category.

The proof of faithfulness of Khovanov and Seidel's categorical representation is based on the fact that there exists a correspondence between curves in a punctured disk and some classes of objects in the category; thus one can use combinatorial and geometric arguments.

\subsection{Extriangulated structures}

It is an intriguing question to ask whether the two representations are related to each other. A hint in that direction is that several gradings can be considered on Khovanov-Seidel's category, and, for some of them, the set of indecomposable modules has the same cardinality as the LKB basis (see~\cite{BravThomas,LQ}) with similar descriptions. The problem is that the filtrations are not preserved by the action of the braid group, and thus one cannot find a unique way to lift the LKB basis. This difficulty is a reflection of a rich interface with algebraic geometry, motivating the study of the braid group action on the moduli space of stability conditions (see~\cite{BDL} and references therein). Importing this notion of stability condition from Bridgeland's work~\cite{Bridgeland} allows to use metric approaches to the study of the action. This perspective will prove very useful to our approach.

Another difficulty that should be mentioned is that a notion of $K_0$ that could be used to relate the categorical representation to the linear one was lacking, notably due to the two variables required by the LKB representation.

In this paper, we use the notion of extriangulated structure to weaken the notion of $K_0$ and obtain something in between the classical Grothendieck group (which is too small) and the split Grothendieck group (which is too big). This algebraic notion, closer to the world of cluster algebras, comes from Nakaoka and Palu~\cite{NP}. One of our main results (Theorem~\ref{thm:extriang}) is that there is a natural way to derive an extriangulated structure on Khovanov-Seidel's category from the choice of a stability condition. This result generalizes to any triangulated category equipped with a central charge under minor assumptions. It is based on the use of a special class of triangles that we call thin since the triangle inequality (using the Bridgeland metric) degenerates at one of their edges. It is interesting to note that the same result recently appeared in~\cite{WWZZ}, with other applications, as we were in the process of writing up our article.

The main theoretical consequence that we state in Theorem~\ref{thm:freeK0}, is that we can describe the extriangulated $K_0$ associated to this extriangulated structure. In the case of Khovanov-Seidel category, the surprise is that this allows for the appearance of a second variable, and the result has the same dimension as the LKB representation.

Unfortunately, there is one such $K_0$ for each stability condition, and the associated extriangulated structures are not preserved by the braid group action. We thus use a system of isomorphisms to rederive the LKB representation from these Grothendieck groups put together (see Theorem~\ref{thm:KhSToLKB}).

\subsection{Perspectives}

The expression of the base change matrices as described in \eqref{eq:alpha_formula} is quite intriguing. It seems that the coefficients of these matrices should, in general, be related to the dimension of some $\Hom$ spaces in a way that we would like to understand. Two other generalizations also directly come to mind. One could extend this towards Artin-Tits groups, using Paris' definition ~\cite{ParisLKB} for the LKB representation outside of type $A$. The most interesting outcome would be to be able to use this strategy for faithfulness questions. Another direction is to rather go towards mapping class groups, the linearity of which is yet to be proved. 

\subsection*{Acknowledgements}

Tony Licata has been closely associated to this project over the last ten years, and we are very grateful for the many ideas he shared over its development. We warmly thank Yann Palu for his patient explanations and his comments on early versions of this paper, including a fix to the proof of Theorem~\ref{thm:extriang}.

\section{Thin triangles}\label{Thin}

\subsection{Triangulated categories}

We will be working in the context of triangulated categories. These categories come with a collection of distinguished triangles and a shift functor $X\mapsto X[1]$, that are subject to a list of axioms~\cite{GelfandManin}. It will be convenient to adopt Dyckerhoff-Kapranov's conventions as in~\cite{DyKa}, where objects are represented by line segments and morphisms by angles. In particular, the octahedral axiom in triangulated categories expresses nicely as a flip in triangulations:
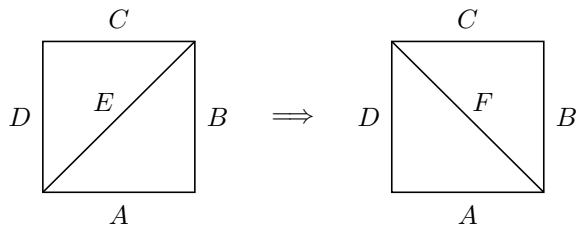
\[
\begin{tikzpicture}[anchorbase]
\draw [semithick] (0,0) rectangle (2,2);
\draw [semithick] (0,0) -- (2,2) ;
\node at (1,-.3) {$A$};
\node at (2.3,1) {$B$};
\node at (1,2.3) {$C$};
\node at (-.3,1) {$D$};
\node at (.8,1.2) {$E$};
\end{tikzpicture}
\quad \Longrightarrow \quad
\begin{tikzpicture}[anchorbase]
\draw [semithick] (0,0) rectangle (2,2);
\draw [semithick] (2,0) -- (0,2);
\node at (1,-.3) {$A$};
\node at (2.3,1) {$B$};
\node at (1,2.3) {$C$};
\node at (-.3,1) {$D$};
\node at (1.2,1.2) {$F$};
\end{tikzpicture}
\]
It will be useful to note that $F$ is constructed as a cone between $B$ and $C$ for example.

\subsection{Extriangulated structure}

The looser notion of extriangulated category will be central in this paper. The notion was introduced by Nakaoka and Palu~\cite{NP}, and we refer to Palu~\cite{Palu} for a survey on the topic. If the category we are working with comes with the structure of a triangulated category, then the definition can be given an easier statement. In particular, the two conditions (ET2) and (ET3) of the abstract definition are automatically satisfied in this setting.

\begin{definition}
An extriangulated structure on a triangulated category $\calC$ is the choice of a set $\Delta$ of exact triangles (called \emph{extriangles}) that respects the following conditions. 
If $X\rightarrow Y \rightarrow Z\xrightarrow{\delta}$ is an extriangle and $f\in \Hom_{\calC}(X,X')$, the triangle
\begin{equation}\tag{ET1}
X'\rightarrow Cone(f\circ \delta)\rightarrow Z\xrightarrow{f\circ \delta}
\label{eq:ET1}
\end{equation}
is an extriangle. Similarly, if $g\in Hom_\calC(Z',Z)$, $X\rightarrow Cone(\delta \circ g) \rightarrow Z'\xrightarrow{\delta\circ g}$ is an extriangle.

Furthermore, if $X$, $X'$, $Y$, $Z$ and $Z'$ are objects in $\calC$ so that the following exact triangles belong to $\Delta$:
\[
X\xrightarrow{f} Y \xrightarrow{f'} Z'\; \text{and} \; Y \xrightarrow{g} Z \xrightarrow{g'} X'
\]
then there exists $Y'\in \calC$ and maps so that the following diagram commutes and all triangles involved are in $\Delta$:
\begin{equation}
\tag{ET4}
\begin{tikzpicture}[anchorbase]
    \node (Xt) at (0,0) {$X$};
    \node (Xb) at (0,-2) {$X$};
    \node (Y) at (2,0) {$Y$};
    \node (Z) at (2,-2) {$Z$};
    \node (X'l) at (2,-4) {$X'$};
    \node (Z') at (4,0) {$Z'$};
    \node (Y') at (4,-2) {$Y'$};
    \node (X'r) at (4,-4) {$X'$};
    \draw [double] (Xt) -- (Xb);
    \draw [double] (X'l) -- (X'r);
    \draw [->] (Xt) to node[midway, above] {$f$} (Y);
    \draw [->] (Y) to node[midway, above] {$f'$} (Z');
    \draw [->] (Xb) to node[midway, above] {$h$} (Z);
    \draw [->] (Z) to node[midway, above] {$h'$} (Y');
    \draw [->] (Y) to node[midway, right] {$g$} (Z);
    \draw [->] (Z) to node[midway, right] {$g'$} (X'l);
    \draw [->] (Z') to node[midway, right] {$d$} (Y');
    \draw [->] (Y') to node[midway, right] {$e$} (X'r);
\end{tikzpicture}
\label{eq:ET4}
\end{equation} 
Similarly, if we have in $\Delta$:
\[X\rightarrow Y \rightarrow Z' \;\text{and}\; X'\rightarrow Z\rightarrow Y\]
then there exists $Y'$ so that the following diagram commutes and all triangles are in $\Delta$:
\begin{equation}
\tag{$\text{ET4}^{\text{op}}$}
\begin{tikzpicture}
    \node (X) at (0,0) {$X$};
    \node (Y) at (2,0) {$Y$};
    \node (Z') at (4,0) {$Z'$};
    \node (Z) at (2,2) {$Z$};
    \node (X') at (2,4) {$X'$};
    \node (Y') at (0,2) {$Y'$};
    \node (Z'2) at (4,2) {$Z'$};
    \node (X'2) at (0,4) {$X'$};
    \draw [->] (X) -- (Y);
    \draw [->] (Y) -- (Z');
    \draw [->] (X') -- (Z);
    \draw [->] (Z) -- (Y);
    \draw [->] (Y') -- (Z);
    \draw [->] (Z) -- (Z'2);
    \draw [->] (X'2)-- (Y');
    \draw [->] (Y') -- (X);
    \draw [double] (X'2) -- (X');
    \draw [double] (Z'2) -- (Z');
\end{tikzpicture}
\label{eq:ET4op}
\end{equation}

\end{definition}

Later, we will exhibit a class of triangles that will be referred to as \emph{thin}. The axioms above can be illustrated in this context as follows. The heuristics here is that if one mutates two triangles that are thin, the resulting triangles are thin too.
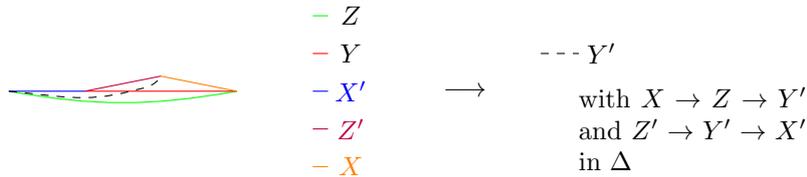
\[
\begin{tikzpicture}[anchorbase]
\draw [green] (0,0) to [out=-10,in=-170] (3,0);
\draw [green] (4,1) -- (4.2,1);
\node at (4.5,1) {$Z$};
\draw [red] (1,0) -- (3,0);
\draw [red] (4,.5) -- (4.2,.5);
\node at (4.5,.5) {$Y$};
\draw [blue] (0,0) -- (1,0);
\draw [blue] (4,0) -- (4.2,0); 
\node [blue] at (4.5,0) {$X'$};
\draw [purple] (1,0) -- (2,.2);
\draw [purple] (4,-.5) -- (4.2,-.5);
\node [purple] at (4.5,-.5) {$Z'$};
\draw [orange] (2,.2) -- (3,0);
\draw [orange] (4,-1) -- (4.2,-1);
\node [orange] at (4.5,-1) {$X$};
\draw [dashed] (0,0) to [out=-5,in=-165] (1.6,0) to [out=15,in=-120] (2,.2);
\node at (6,0) {$\longrightarrow$};
\draw [dashed] (7,.5) -- (7.5,.5);
\node at (7.8,.5) {$Y'$};
\node at (9,-.5) {\begin{minipage}{3cm} with $X\rightarrow Z \rightarrow Y'$ and $Z'\rightarrow Y' \rightarrow X'$ in $\Delta$ \end{minipage}};
\end{tikzpicture}
\]

\subsection{Stability conditions}

In order to pick out sets $\Delta$ to define extriangulated structures, we will use Bridgeland's notion of stability conditions on triangulated categories~\cite{Bridgeland}. Recall the definition of a stability condition.

\begin{definition}
  Let $\calC$ be a triangulated category. A stability condition $\tau$ is a pair $(\mathcal{P},Z)$ with:
  \begin{itemize}
  \item $\mathcal{P}$ a slicing \textit{i.e.} a full additive subcategory $\mathcal{P}(\phi)$ for all $\phi\in \R$; objects of these subcategories are called \emph{semi-stable}; semi-stable objects that are simple are called \emph{stable}; $\phi$ is called the phase of any object in $\mathcal{P}(\phi)$;
  \item $Z$ a central charge, which is a group homomorphism $Z:K_0(\calC)\mapsto \C$.
  \end{itemize}
 Moreover $\mathcal{P}$ and $Z$ are subject to the following conditions:
  \begin{enumerate}
  \item $\mathcal{P}(\phi+1)=\mathcal{P}(\phi)[1]$;
  \item if $A\in \calP(\phi)$ and $B\in \calP(\psi)$ with $\phi>\psi$ then $\Hom(A,B)=0$;
  \item for non-zero object $X\in \calC$, there exists a filtration (called \emph{Harder-Narasimhan filtration}, or HN-filtration for short):
    \[
    \begin{tikzpicture}
      %Top row
      \node (T0) at (0,0) {$0=X_0$};
      \node (T1) at (2,0) {$X_1$};
      \node (T2) at (4,0) {$X_2$};
      \node (Tdots) at (5,0) {$\cdots$};
      \node (Tn-1) at (6,0) {$X_{n-1}$};
      \node (Tn) at (8,0) {$X_n=X$};
      % Bottom row
      \node (B1) at (1,-2) {$Y_1$};
      \node (B2) at (3,-2) {$Y_2$};
      \node (Bdots) at (5,-2) {$\cdots$};
      \node (Bn) at (7,-2) {$Y_n$};
      % horizontal arrow
      \draw [->] (T0) -- (T1);
      \draw [->] (T1) -- (T2);
      \draw [->] (Tn-1) -- (Tn);
      % back arrows downwards
      \draw [->] (T1) -- (B1);
      \draw [->] (T2) -- (B2);
      \draw [->] (Tn) -- (Bn);
      % Upward dashed arrows
      \draw [dashed, ->] (B1) -- (T0) node [midway,xshift=-.2cm,yshift=-.1cm] {\tiny $+1$};
      \draw [dashed, ->] (B2) -- (T1) node [midway,xshift=-.2cm,yshift=-.1cm] {\tiny $+1$};
      \draw [dashed, ->] (Bn) -- (Tn-1) node [midway,xshift=-.2cm,yshift=-.1cm] {\tiny $+1$};
    \end{tikzpicture}
    \]
    where all triangles are exact, $Y_i\in \calP(\phi_i)$ and $\phi_1>\phi_2>\dots>\phi_n$;
  \item $Z\left(\calP(\phi)\setminus \{0\}\right)\subset \R^{+,*}e^{i\pi \phi}$.
  \end{enumerate}
  \end{definition}
  
The slices $\mathcal{P}([0,1))$ form the heart of a preferred $t$-structure associated to the slicing.

The set of stability conditions on $\calC$ will be denoted $Stab(\calC)$. There is a natural action of $\C^\ast$ on $Stab(\calC)$ given, on the central charge $Z$, and, on the slicing, by translation by the phase. The quotient will be denoted $PStab(\calC)=Stab(\calC)/\C^\ast$.

Following Bridgeland, one can define a metric on objects (called the \emph{mass} or $\tau$-\emph{mass}) by adding up the norm of the images under the central charge of all elements in the HN-filtration
\begin{equation} \label{eq:defm}
m_\tau(X)=\sum_{X_i\in HN(X)} |Z(Y_i)|.
\end{equation}
The mass of the $0$ object is set to be $0$.

\subsection{Thin triangles make an extriangulated structure}

In the next section, we will only use the existence of the metric $m_\tau$, which needs not come from a stability condition. We suppose given a metric $m_{\tau}$ that is additive under direct sum, invariant under shift and respects the triangle inequality on exact triangles. We then consider the following collection $\Delta_{\tau}$ of what we will refer to as thin triangles
\[
\Delta_{\tau}=\{
A\rightarrow B \rightarrow C\;|\;A,B,C\in \calC, \; m_{\tau}(B)=m_\tau(A)+m_\tau(C).
\}
\]

Here notice that this set of triangles is not stable under triangulated shifts: we impose that the object in the middle of the triangle is the long one. 

\begin{remark} \label{rem:idtriangle}
A key remark is that, unless $X$ is the zero object, the following triangle is \emph{not} thin
\[
X[1]\rightarrow Cone(\id_X) \rightarrow X.
\]
Indeed, the middle object is homotopic to zero, so its mass is null. The two rotations of this exact triangle would be thin, on the other hand. This will play a key role in endowing the Grothendieck group with a module structure over a polynomial ring, as we will see later when we recover the LKB representation of braids.
\end{remark}

The following is the main result of the first part of our paper.

\begin{theorem} \label{thm:extriang}
Let $\calC$ be a triangulated category with a Bridgeland stability condition $\tau$. Then the associated set of thin triangles $\Delta_{\tau}$ endows $\calC$ with an extriangulated structure.
\end{theorem}

\begin{proof}
We first check \eqref{eq:ET1}. We consider the following exact triangle
\[
\begin{tikzpicture}[anchorbase]
\node (X) at (0,0) {$X$};
\node (Xp) at (3,.5) {$X'$};
\node at (3,0) {$\oplus$};
\node (ZX) at (3,-.5) {$Cone(Z\rightarrow X$)};
\node (ZXp) at (6,0) {$Cone(Z\rightarrow X').$};
\draw [->] (X) -- (Xp);
\draw [->] (X) -- (ZX);
\draw [->] (Xp) -- (ZXp);
\draw [->] (ZX) -- (ZXp);
\end{tikzpicture}
\]
Since this is an exact triangle, one reads
\[
m_{\tau}(X')+m_{\tau}(Z\rightarrow X)\leq m_{\tau}(X)+m_{\tau}(Z\rightarrow X').
\]
Now, one can use that $m_{\tau}(Z\rightarrow X)=m_{\tau}(X)+m_{\tau}(Z)$ (since the corresponding triangle is thin) to simplify the above inequality and get
\[
m_{\tau}(X')+m_{\tau}(Z) \leq m_{\tau}(Z\rightarrow X').
\]
Combined with the triangle inequality associated to the exact triangle $X'\rightarrow Cone(Z\rightarrow X')\rightarrow Z$, one reads
\[
m_{\tau}(Z\rightarrow X')=m_{\tau}(Z)+m_{\tau}(X')
\]
as desired. The other version of \eqref{eq:ET1} holds for similar reasons. 

Let us consider \eqref{eq:ET4} and triangles $X\rightarrow Y \rightarrow Z'$ and $Y\rightarrow Z \rightarrow X'$ that belong to $\Delta_{\tau}$. We have
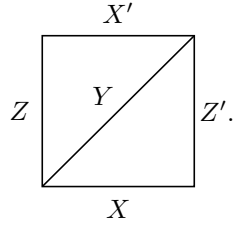
\[
\begin{tikzpicture}[anchorbase]
\draw [semithick] (0,0) rectangle (2,2);
\draw [semithick] (0,0) -- (2,2) ;
\node at (1,-.3) {$X$};
\node at (2.3,1) {$Z'.$};
\node at (1,2.3) {$X'$};
\node at (-.3,1) {$Z$};
\node at (.8,1.2) {$Y$};
\end{tikzpicture}
\]
Using the octahedral axiom, we get the existence of $Y'$, the flip of $Y$, giving the commutative diagram from (ET4)
\[
\begin{tikzpicture}[anchorbase]
\node (Xt) at (0,0) {$X$};
\node (Y) at (2,0) {$Y$};
\node (Zp) at (4,0) {$Z'$};
\node (Xb) at (0,-2) {$X$};
\node (Z) at (2,-2) {$Z$};
\node (Yp) at (4,-2) {$Y'$};
\node (Xpm) at (2,-4) {$X'$};
\node (Xpr) at (4,-4) {$X'.$};
\draw [->] (Xt) -- (Y);
\draw [->] (Y) -- (Zp);
\draw [double] (Xt) -- (Xb);
\draw [->] (Y) -- (Z);
\draw [->] (Xb)-- (Z);
\draw [->] (Z) -- (Yp);
\draw [->] (Zp) -- (Yp) ;
\draw [double] (Xpm) -- (Xpr);
\draw [->] (Z) -- (Xpm);
\draw [->] (Yp) -- (Xpr);
\end{tikzpicture}
\]
We want to show that the triangles $X\rightarrow Z \rightarrow Y'$ and $Z'\rightarrow Y'\rightarrow X'$ are thin.

Because these are triangles, we have
\begin{gather}
\label{eq:ineqthin}
m_\tau(Z)\leq m_\tau(X)+m_\tau(Y') \\
m_\tau(Y')\leq m_\tau(Z')+m_\tau(X'). \nonumber
\end{gather}
These combine into
\[
m_\tau(Z)\leq m_\tau(X)+m_\tau(Z')+m_\tau(X')=m_\tau(Y)+m_\tau(X').
\]
The equality on the right holds because $X\rightarrow Y \rightarrow Z'$ is thin. But the terms on the LHS and RHS above are actually equal, since $Y\rightarrow Z \rightarrow X'$ is thin. Thus the two inequalities in \eqref{eq:ineqthin} actually are equalities, and hence the two triangles are thin and axiom (ET4) holds.

Axiom \eqref{eq:ET4op} holds for similar reasons.
\end{proof}

\begin{remark}
From the proof it follows that the intersection $\bigcap_{k=1}^n\Delta_{\tau_k}$ also induces an extriangulated structure. More generally, Theorem~\ref{thm:extriang} holds for metrics that are additive under direct sum and satisfy triangle inequalities for exact triangles.
\end{remark}

\subsection{Another extriangulated structure}

Theorem~\ref{thm:extriang} as well as the ideas developed for the proof of Theorem~\ref{thm:freeK0} are close to the notion of \emph{rectifiable filtration} from~\cite[Appendix A]{BDL} (see in particular Proposition A.9). A more diagrammatic presentation for this notion appears in the appendix of Heng's PhD thesis~\cite{HengPhD}. We will quickly review this notion and see that it provides an alternative way to produce extriangulated structures on a triangulated category from the choice of a stability condition.

Let us fix $\tau$ a stability condition on $\calC$. Recall the following definition.
\begin{definition}
A distinguished triangle $A\rightarrow B \rightarrow C\xrightarrow{\delta}$ is rectifiable if, $\forall \phi \in \R$, the following maps compose to zero
\[
C_{>\phi} \rightarrow C \xrightarrow{\delta} A[1]\rightarrow A_{\leq \phi}[1],
\]
where $C_{>\phi}$ denotes the maximal submodule of $C$ belonging to $\calP((\phi, +\infty))$ and $A_{\leq \phi}$ denotes the maximal quotient of $A$ belonging to $\calP((- \infty,\phi])$.
\end{definition}
The incentive for this definition is that, in such a case, the stables that appear in the HN-filtration of $B$ are the union of those of $A$ and those of $C$.

We can thus introduce the set $\Delta_{\tau}^r$ of rectifiable triangles as follows
\[
\Delta_{\tau}^r=\{
A\rightarrow B \rightarrow C\; \text{rectifiable},\; \; A,B,C\in \calC 
\}.
\]

It turns out that these give rise to an extriangulated structure.

\begin{proposition}
The set of triangles $\Delta_{\tau}^r$ gives rise to an extriangulated structure.
\end{proposition}
\begin{proof}
All axioms have already been checked: \eqref{eq:ET1} and its opposite versions follow from \cite[Proposition A.4]{BDL} or \cite[Proposition B.0.4]{HengPhD}, and \eqref{eq:ET4} and \eqref{eq:ET4op} follow from \cite[Proposition~B.0.8]{HengPhD}.
\end{proof}

\begin{remark}
In finite type, for generic stability conditions, $\Delta_{\tau}$ and $\Delta_{\tau}^r$ should yield the same structure. However, they have inverse behaviors on walls: there are more thin triangles on a wall, but less rectifiable triangles.
\end{remark}

\subsection{Extriangulated Grothendieck group}

Extriangulated structures on categories lead to the notion of extriangulated Grothendieck groups, obtained by linearizing only thin triangles, as follows.

\begin{definition}[\protect{\cite[Definition 4.8]{PPPP}}]
The extriangulated Grothendieck group associated to the choice $\Delta$ of a set of triangles is the quotient of the free abelian group generated by isomorphism classes $[X]$, for each $X\in \calC$, by the relations $[X]-[Y ]+[Z]$, for all triangles $X\rightarrow Y \rightarrow Z$ in $\Delta$.
\end{definition}

For example, in type $A_n$ for Khovanov-Seidel, we will see (Corollary~\ref{cor:freeK0} ) that we get a free $\Z[q^{\pm 1},t^{\pm 1}]$-module of rank $\frac{n(n+1)}{2}$ for a generic stability condition, instead of rank $n$ over $\Z[q^{\pm 1}]$ for the usual Grothendieck group.

We will state the next theorem in general, before specializing it to the case of the braid group (see Corollary~\ref{cor:freeK0}).

In the next theorem, we require that the masses of the stable objects are linearly independent. In the case of finite type $A$ that we will be interested in later on, this is only for convenience as one could simply perturb a little bit the stability condition without changing the set of thin triangles if this turned out not to be the case.

\begin{theorem}\label{thm:freeK0}
Let $\calC$ be a triangulated category together with given stability condition $\tau$ and collection of thin triangles $\Delta_{\tau}$. Assume that the modulus of the central charges (or masses) 
of the stable objects are independent over $\Z$. Then the extriangulated Grothendieck group associated to the stability condition $\tau$ and denoted $K_0^{\tau}(\calC)$
 is freely generated over $\Z[s^{\pm 1}]$ by the isomorphism classes of the stable objects that are in the heart of the preferred $t$-structure associated to $\tau$. The variable $s$ accounts for the triangulated shift.
\end{theorem}

The proof of Theorem~\ref{thm:freeK0} will be based on Lemma \ref{lem:tr}, in which we will make use of the notion of composition series in abelian categories. To explain the idea, consider an object $X\in \calC$. This object can be cut into slices, thanks to its Harder-Narasimhan filtration. Then each semi-stable object $Y_i$ appearing in the HN-filtration is an object in an abelian category $\calP(\phi)$ (see~\cite[Lemma 5.2]{Bridgeland} or~\cite[Exercise 5.9]{MacriSchmidt}), which itself can be decomposed into simple (or equivalently stable) objects, in an essentially unique way thanks to the Jordan-Hölder theorem. Given $Y_i$, we will denote $JH(Y_i)$ the collection of simples with their multiplicities that compose $Y_i$. Now, to $X$, we can associate $JH(HN(X))$, the set of simples that compose the semi-stables appearing in the HN-filtration of $X$, counted with multiplicities.

\begin{lemma} \label{lem:tr}
Let $A\rightarrow B \rightarrow C$ be a thin triangle. We have
\[
JH(HN(B))=JH(HN(A))\cup JH(HN(C)).
\]
In other words, the union of the JH decompositions of all terms in the HN-filtration of $B$ agrees with the union of the JH decompositions of the stables in the HN-filtrations of $A$ and $C$.
\end{lemma}
\begin{proof}
Notice that in each abelian category $\calP(\phi)$, the simple objects are the stable objects.

Start from $m_\tau(B)=m_\tau(A)+m_\tau(C)$. For $X\in \{A,B,C\}$, note that
\begin{equation} \label{eq:multiplicities}
m_\tau(X)=\sum x_{i} m_\tau(S_i)
\end{equation}
where the $S_i$'s are all stables. This is the result of first computing the HN-filtration of $X$ and using the identity from Equation~\eqref{eq:defm}. Then each semi-stable object can be further decomposed into simples by taking a composition series.
Notice that in Equation~\eqref{eq:multiplicities}, for all $i$, $x_i$ is non-negative and counts the number of times each $S_i$ appears. Furthermore, because we have assumed $\Z$-linear independence, the value $m_\tau(X)$ determines the $x_i$'s.

The condition of the triangle being thin thus re-expresses as
\[b_i=a_i+c_i,\; \forall i.\]
This implies the result.
\end{proof}

\begin{proof}[Proof of Theorem~\ref{thm:freeK0}]
Any object  $X$ in $\calC$ has an Harder-Narasimhan filtration induced by $\tau$, with semi-stables $Y_i$, as follows
    \[
    \begin{tikzpicture}
      %Top row
      \node (T0) at (0,0) {$0=X_0$};
      \node (T1) at (2,0) {$X_1$};
      \node (T2) at (4,0) {$X_2$};
      \node (Tdots) at (5,0) {$\cdots$};
      \node (Tn-1) at (6,0) {$X_{n-1}$};
      \node (Tn) at (8,0) {$X_n=X$};
      % Bottom row
      \node (B1) at (1,-2) {$Y_1$};
      \node (B2) at (3,-2) {$Y_2$};
      \node (Bdots) at (5,-2) {$\cdots$};
      \node (Bn) at (7,-2) {$Y_n$};
      % horizontal arrow
      \draw [->] (T0) -- (T1);
      \draw [->] (T1) -- (T2);
      \draw [->] (Tn-1) -- (Tn);
      % back arrows downwards
      \draw [->] (T1) -- (B1);
      \draw [->] (T2) -- (B2);
      \draw [->] (Tn) -- (Bn);
      % Upward dashed arrows
      \draw [dashed, ->] (B1) -- (T0) node [midway,xshift=-.2cm,yshift=-.1cm] {\tiny $+1$};
      \draw [dashed, ->] (B2) -- (T1) node [midway,xshift=-.2cm,yshift=-.1cm] {\tiny $+1$};
      \draw [dashed, ->] (Bn) -- (Tn-1) node [midway,xshift=-.2cm,yshift=-.1cm] {\tiny $+1$};
    \end{tikzpicture}
    \]
We first want to see that $[X]\in K_0^{\tau}(\calC)$ expresses as a linear combination of the semi-stables $Y_i$'s.

To see this, consider the triangle $X_{n-1}\rightarrow X_n \rightarrow Y_n$. It follows from the definition of $m_\tau$ that this triangle is thin, thus
\[
\left[X\right]=\left[X_n\right]=\left[X_{n-1}\right]+\left[Y_n\right].
\]
The object $Y_n$ is semi-stable and we can conclude by induction that $\left[X_{n-1}\right]$ is a linear combination of semi-stables. But then semi-stables can themselves be decomposed into stables through a sequence of thin triangles. Indeed, a semi-stable object $B$ is one that fits in a triangle $A\rightarrow B\rightarrow C$ with $A$, $B$, $C$ all living in the same $\calP(\phi)$. It follows from there that $Z(B)=Z(A)+Z(C)$ and, because they all have same phase, that $m_\tau(B)=m_\tau(A)+m_\tau(C)$. So any object expresses as a linear combination of stable objects. Furthermore, the $Y_i$'s all express as shifts of semi-stable objects with phase in $[0,1)$. It follows that the shifts of the stables in the heart span the extriangulated Grothendieck group.

Let us now fix the list $\{S_i\}$ of all stables in the heart of the slicing.
We want to show that all triangulated shifts of these stables freely generate $K_0^{\tau}(\calC)$ as a $\Z$-module. To this aim, let  $M$ be the free $\Z$-module formally generated by all shifts of these stables
\[
M=\Z\langle \left[S_i[l]\right] \rangle.
\]

Recall that $K_0^{\mathrm add}(\calC)$, the additive Grothendieck group of $\calC$, is the free $\Z$-module generated by isomorphism classes of objects modulo the relations $[A\oplus B]=[A]+[B]$. One can define a map from $K_0^{\mathrm add}(\calC)$ to $M$ as follows
\begin{align*}
\psi: K_0^{\mathrm add}(\calC)\quad &\rightarrow \quad M \\
X\quad &\mapsto \sum_{W\in JH(HN(X))} [W].
\end{align*}
It is obtained as in Lemma~\ref{lem:tr} by first taking the formal combination of the semi-stables appearing in the HN-filtration, and then replacing semi-stables by their composition series. Lemma~\ref{lem:tr} ensures that the map $\psi$
factors through $K_0^\tau(\calC)$. Hence the stable objects are linearly independent in $K_0^\tau(\calC)$.

As a consequence, $K_0^{\tau}(\calC)$ has a structure of free $\Z[s^{\pm 1}]$-module with basis $\{S_i\}$. Here the multiplication by $s$ corresponds to the induced action of the triangulated shift endofunctor.
\end{proof}

\subsection{Linear representations from categorical actions?} \label{sec:linearrep}

A priori, there is no reason why a set of thin triangles should be preserved under the action of $AutEq(\calC)$. However, one can fix a base point $\tau_0$ in $Stab(\mathcal{C})$, the space of stability conditions, and consider the orbit $T$ of $\tau_0$ under the action of autoequivalences (later we will take a further quotient of $T$).

In order to allow for more flexibility, we will give a definition in a slightly more general setting.

\begin{definition}
Let $G$ be a group acting on the triangulated category $\calC$. We fix a stability condition $\tau_0$ together with its set $\Delta_{\tau_0}$ of thin triangles. Let T be $G\cdot \tau_0$ and $\tilde{T}=T/\sim$ its quotient under some equivalence relation $\sim$.
An identification system for $(G,\tilde{T})$ is a choice of an isomorphism of $\Z$-modules $M_{\tau}:K_0^{\tau_0}(\calC)\simeq K_0^{\tau}(\calC)$ for any $\tau\in \tilde{T}$.
\end{definition}

Later, we will consider stability conditions up to the $\C^*$-action and up to identification of generic conditions that share the same stables.

From such a system and from a choice of basis of each extriangulated Grothendieck group, one can try to get a matrix representation of elements of $G$ by considering
\[
\begin{tikzpicture}[anchorbase]
\node (0) at (0,0) {$K_0^{\tau_0}(\calC)$};
\node (1) at (-1,2) {$K_0^{\tau}(\calC)$};
\node (2) at (1,2) {$K_0^{g\tau}(\calC)$};
\draw [->] [bend left] (1) to node [midway,above] {$g$} (2);
\draw [green,->] (0) to node [midway,left] {$M_{\tau}$} (1);
\draw [green,->] (0) to node [midway,right] {$M_{g\tau}$} (2);
\draw [red,->] (0) to [out=110, in=180] (0,2) to [out=0,in=70] (0);
\node [red] at (0,1.2) {$\rho(g)$};
\end{tikzpicture}
\]
Above we abuse notation when writing $K_0^\tau(\calC)$: we identify $\tau$ and its class in $\tilde{T}$.

\begin{definition}\label{def:compatible}
The identification system is said to be compatible if $\rho(g)$ depends only on $g$ (and not on $\tau$).
\end{definition}

Then we get a well-defined representation on $K_0^{\tau_0}(\calC)$.

\begin{question}
When does there exist such a compatible system? When does it give something interesting?
\end{question}

We develop a first answer to this question by showing that the LKB representation can arise in this way from the Khovanov-Seidel categorical action.

If one works with $T$ without further restrictions, then we might have too much freedom as stated in the following lemma.

\begin{lemma} \label{lem:triv}
Assume that a group $G$ acts freely on a stability condition $\tau_0$, and that $K_0^{\tau_0}(\calC)$ is a free module of finite rank $r$ over some commutative ring $R$. Then any $R$-linear representation of $G$ of rank~$r$ can be recovered from an identification system.
\end{lemma}

\begin{proof}
Choose arbitrarily $M_{\tau_0}$ an invertible matrix representing an endomorphism of $K_0^{\tau_0}(\calC)$. Denote $\rho$ the given representation that we, by abuse of notation, see as
\[
\rho: G\rightarrow End(K_0^{\tau_0^{\mathrm{ref}}}).
\]
where $\tau_0^{\text{ref}}$ is nothing but $\tau_0$ and this notation is used to highlight the fact that $M_{\tau_0}$ may not be the identity matrix.

For any $g\in G$ and $\tau \in T$, we denote $P_\tau(g)$ the matrix obtained from the induced action of $g$ from $K_0^{\tau}(\calC)$ to $K_0^{g\tau}(\calC)$.
Notice that for any $g,g'\in G$ and $\tau \in T$, we have $P_\tau(g'g)=P_{g\tau}(g')P_\tau(g)$.
Since the action is free, for any $\tau \in T$, there exists a unique $g_{\tau}\in G$ such that $\tau=g_{\tau}\tau_0$. Then there exists a unique matrix $M_{\tau}$ determined by
\[
M_{\tau}^{-1}P_{\tau_0}(g_{\tau})M_{\tau_0}=\rho(g_{\tau}).
\] 

The process can be depicted as follows

\begin{equation} \label{eq:protosystem0}
\begin{tikzpicture}[anchorbase]
\node [circle, draw] (ref) at (0,0) {$\tau_0^{\text{ref}}$};
\node [circle, draw] (T0) at (0,4) {$\tau_0$};
\node [circle, draw] (T1) at (2.8,2.8) {$\tau$};
\draw [->] (ref) to [out=100,in=-100] node [midway,left] {\small $M_{\tau_0}$} (T0);
\draw [->] (T0) to  [out=0,in=135] node[midway, above] {$P_{\tau_0}(g_{\tau})$}  (T1);
\draw [->] (ref) to [out=90,in=145] (.5,1) node [above] {$\rho(g_{\tau})$} to [out=-35,in=40] (ref);
\draw [dashed,->] (ref) to [out=30,in=-100] node [midway,right] {$M_{\tau}$} (T1);
\end{tikzpicture}
\end{equation}

It remains to check that, for all $g\in G$, $M_{g\tau}^{-1}P_{\tau}(g)M_{\tau}$ does not depend on $\tau$.
We have

\begin{align*}
    M_{g\tau}^{-1}P_{\tau}(g)M_{\tau}&=M_{gg_{\tau}\tau_0}^{-1}P_{g_{\tau}\tau_0}(g)M_{g_{\tau}\tau_0}\\&=\rho(gg_{\tau})M_{\tau_0}^{-1}P_{\tau_0}(gg_{\tau})^{-1}P_{g_{\tau}\tau_0}(g)P_{\tau_0}(g_{\tau})M_{\tau_0}\rho(g_{\tau})^{-1}\\&=\rho(g)\rho(g_{\tau})M_{\tau_0}^{-1}P_{\tau_0}(gg_{\tau})^{-1}P_{\tau_0}(gg_{\tau})M_{\tau_0}\rho(g_{\tau})^{-1}\\&=\rho(g).
\end{align*}

\end{proof}

\begin{remark}
Lemma~\ref{lem:triv} seems to be giving quite a lot of freedom to recover representations from categorical actions. In our context, it should be noticed that the first challenge was to identify a decategorification procedure that would allow to obtain a 2-variable representation. This motivated the connection with extriangulated structures, and is a key ingredient to make the process work. Furthermore, the way we extract the LKB representation from the categorical data preserves several desirable symmetries: it factors through $PStab$, and also two stability conditions that share the same combinatorics will be identified (this will be the equivalence relation $\sim$). This in particular implies that the (dual) Garside element preserves the equivalence class of the chosen reference stability condition. The compatibility between the categorical representation and the LKB representation that ensures these symmetries turned out to require quite some work -- see Theorem~\ref{thm:KhSToLKB} and its proof. Finally, the shape of the identification system makes us hope that the link between the two representations will help identifying new combinatorial structures, see Remark~\ref{rem:matrices}.
\end{remark}

\section{Khovanov-Seidel and LKB braid representations}

\subsection{Khovanov-Seidel representation} \label{sec:KhS}

\subsubsection{The category} 

We want to use the extriangulated structures inhereted from stability conditions on a category introduced by Khovanov and Seidel~\cite{KhS}. This category is endowed with an action of the braid group, the decategorification of which recovers the Burau representation~\cite{Burau}.

The category is presented as the homotopy category of $\Z-$graded projective modules over the so-called zig-zag algebra associated to the Dynkin diagram of type $A_n$. This construction has been extended to any simply-laced Artin-Tits group, and widely studied since then.

In view of geometric group theory, understanding the structure of the moduli space of Bridgeland stability conditions and the action it carries, is an important question (see~\cite{BDL} and references therein). Here we will strive to only partly linearize the Khovanov-Seidel action by use of stability conditions.

In type $A_n$, the zig-zag algebra $\zz{A}_n$ is a quotient of the path algebra associated to the doubled quiver, as outlined below
\begin{gather*}
% Type A Dynkin diagram:
\begin{tikzpicture}[anchorbase]
\node at (2,.6) {$\Gamma_{A_n}$};
\node at (2,-.6) {\vphantom{$\Gamma_{A_n}$}};
\node (1) at (1,0) {$\bullet$};
\node (2) at (2,0) {$\bullet$};
\node (3) at (3,0) {$\bullet$};
\node (n) at (4.2,0) {$\bullet$};
\node at (1) [below] {\small $1$};
\node at (2) [below] {\small $2$};
\node at (3) [below] {\small $3$};
\node at (n) [below] {\small $n$};
\node (dots) at (3.6,0) {$\cdots$};
\draw (1) -- (2);
\draw (2) -- (3);
\end{tikzpicture}
\quad \longrightarrow \quad
% Type A doubled quiver:
\begin{tikzpicture}[anchorbase]
\node at (2,.6) {${\vec\Gamma_{A_n}}$};
\node at (2,-.6) {\vphantom{$\Gamma_{A_n}$}};
\node (1) at (1,0) {$\bullet$};
\node (2) at (2,0) {$\bullet$};
\node (3) at (3,0) {$\bullet$};
\node (n) at (4.2,0) {$\bullet$};
\node at (1) [below] {\small $1$};
\node at (2) [below] {\small $2$};
\node at (3) [below] {\small $3$};
\node at (n) [below] {\small $n$};
\node (dots) at (3.6,0) {$\cdots$};
\draw[->] (1) to [out=20,in=160] (2);
\draw[<-] (1) to [out=-20,in=-160] (2);
\draw[->] (2) to [out=20,in=160] (3);
\draw [<-] (2) to [out=-20,in=-160] (3);
\end{tikzpicture}
\\
\zz{A}_n:=
  \frac{
    \mathrm{Path}({\vec\Gamma_{A_n}})
    }{
    \begin{tikzpicture}[anchorbase, scale=.5]
    \node (1) at (1,0) {$\bullet$};
    \node (2) at (2,0) {$\bullet$};
    \node (3) at (3,0) {$\bullet$};
    \draw [opacity=.5] (1.center) -- (2.center) -- (3.center); \draw [->] (1.north) to [out=30,in=150] (2.north) to [out=30,in=150] (3.north); \end{tikzpicture} =0\;,\quad \begin{tikzpicture}[anchorbase] \node (1) at (1,0) {$\bullet$}; \node (2) at (2,0) {$\bullet$}; \node (3) at (3,0) {$\bullet$}; \draw [opacity=.5] (1.center) -- (2.center) -- (3.center); \draw [->] (2.center) to [out=120,in=60] (1.east) to [out=-60,in=-120] (2.west); \end{tikzpicture}=\begin{tikzpicture}[anchorbase] \node (1) at (1,0) {$\bullet$}; \node (2) at (2,0) {$\bullet$}; \node (3) at (3,0) {$\bullet$}; \draw [opacity=.5] (1.center) -- (2.center) -- (3.center); \draw [->] (2.center) to [out=60,in=120] (3.west) to [out=-120,in=-60] (2.east); \end{tikzpicture}
    }
\end{gather*}

The collection $\{i\}$ of vertices induces a set of primitive idempotents $\{e_i\}$, which in turn induces a complete collection $\{P_i=\zz{A}_ne_i\}$ of indecomposable left-projectives, and the $\Z-$grading is given by the path length. The path length grading shift by $k\in \Z$ is denoted by $\langle k \rangle$, while the homological shift by $l\in \Z$ is denoted $\{l\}$.

\begin{definition}
The main category at play in this paper is the homotopy category of graded, projective $\zz{A}_n$-modules, denoted $\calC(A_n)$ (or simply $\calC$).
\end{definition}
We will give more details on the gradings in Section \ref{sec:2varK0}.

\subsubsection{The action}

This category can be endowed with an action of the braid group $B_{n+1}$ on $n+1$ strands by autoequivalences. These auto-equivalences are generated by spherical twists over the $P_i$'s (see~\cite{KhS} and \cite{BravThomas} for example). Those are given by tensoring with the following complexes of bimodules
  \begin{align}
    \Sigma_i&= \zz{A}_ne_i\otimes_{\C} e_i\zz{A}_n\{-1\} \xrightarrow{f} \zz{A}_n \\
    \Sigma_i^{-1}&= \zz{A}_n \xrightarrow{g} {\bf A}_ne_i\otimes_\C e_i\zz{A}_n\langle -2\rangle\{1\} .
  \end{align}
The map $f$ is given by multiplication, and the map $g$ is its adjoint and can be determined explicitly.

The main result in~\cite{KhS} is that this produces a faithful action of the braid group.

\subsection{The LKB representation} \label{sec:LKBdef}

At about the same time that Khovanov and Seidel solved the faithfulness issue of the Burau representation by upgrading it to a categorical representation, Krammer~\cite{Krammer} and Bigelow~\cite{Bigelow_linearity}, building upon work of Lawrence~\cite{Lawrence}, proved the faithfulness of a deformation of the second symmetric power of the Burau representation, now known as the Lawrence-Krammer-Bigelow (LKB) representation -- thereby showing that braid groups are linear.

The main application of the tools we develop in this paper is to show that the LKB representation can be derived from the Khovanov-Seidel categorical action, see Theorem~\ref{thm:KhSToLKB}.

To work with the LKB representation, we follow the conventions of Paoluzzi and Paris~\cite{PaPa}.

\begin{definition} \label{def:LKB}
    The LKB representation of $B_{n+1}$ is given by the action on a free $\Z[x,y]$-module of rank $\frac{n(n+1)}{2}$, spanned by vectors $e_{ij}$ with $1\leq i<j\leq n+1$. The action is given by
    \[
  \sigma_k\cdot e_{ij}=
  \left\{
 \begin{alignedat}{2}
 & xe_{i-1j}+(1-x)e_{ij} && \quad \text{if} \; k=i-1 \\
 & e_{i+1j}-xy(x-1)e_{kk+1} && \quad \text{if} \;k=i<j-1 \\
 & -x^2ye_{kk+1} && \quad \text{if} \; k=i=j-1 \\
 & e_{ij}-y(x-1)^2e_{kk+1} &&\quad \text{if} \; i<k<j-1 \\
 & e_{ij-1}-xy(x-1)e_{kk+1} && \quad \text{if} \; i<j-1=k \\
 & xe_{ij+1}+(1-x)e_{ij} && \quad \text{if}\; k=j \\
 & e_{ij} && \quad \text{otherwise.}
 \end{alignedat} 
  \right.
    \]
\end{definition}

\section{Recovering the LKB representation}

In this section, we fully develop, in the case of Khovanov-Seidel zig-zag algebras, the construction outlined in Section \ref{Thin} in order to recover the LKB representation.

\subsection[A two-variable K0]{A two-variable \texorpdfstring{$K_0$}{K0}} \label{sec:2varK0}

We start by fixing $\tau_0$, a preferred stability condition on $\calC(A_n)$. Our choice for $\tau_0$ goes as follows. We start by declaring the collection of stables. These will be objects that admit a minimal complex where all differential maps are of the kind $P_i\rightarrow P_{i+1}\{1\}\langle-1\rangle$. It follows from~\cite[Theorem 2.7]{LQ} that these given stables are determined by a shift and a positive root in the root system associated to the type $A$ Dynkin diagram. In practice, ignoring shifts, these will be complexes determined by a pair $1\leq i<j \leq n+1$:
\[
P_i\rightarrow P_{i+1}\rightarrow \cdots \rightarrow P_{j-1}.
\]

Making the definition more precise would require to order the stables by increasing phase, and then to pick precise values for their images under $Z$. There are nice picture-based descriptions for such choices that can be found in~\cite[Section 6]{BDL25}. For our purpose, let us simply admit that such a choice exists.

For later use, note that this choice of $\tau_0$ is related to the choice of a Coxeter element in the symmetric group, which itself is equivalent to fixing an orientation of the Dynkin diagram. Our choice is to have the Dynkin diagram of type $A$ oriented from left to right, which corresponds to choosing the Coxeter element $c=s_{1}\cdots s_n$. This choice lifts in the braid group to the so-called (dual) Garside element $\gamma=\sigma_{1}\cdots \sigma_n$, and a key feature is that $\gamma$ will preserve the collection of stables associated to $\tau_0$ (see for example Equation~\ref{eq:action_partial_gamma}).

In addition, we assume that $\tau_0$ is chosen so that the images of the stables under the central charge $Z$ are linearly independent over $\Z$. This is a technical condition that generically holds, and that will ensure that the mass of an object determines its HN stables.

As in Section~\ref{sec:linearrep}, we first consider
\[
T = B_{n+1}\cdot \tau_0,
\]
the orbit of $\tau_0$ under the action of the braid group. It should be recalled that the action of the braid group on the moduli space of stability conditions is free~\cite[Proposition 7.25]{HengLicata}, but not transitive. Indeed, the action of the braid group will just change the stable objects, but the set of phases for the stables will remain unchanged. A fundamental domain for the action naturally splits into several combinatorial classes, separated by walls where the phases of some of the stables align. This is already visible in type $A_3$. The reader familiar with the correspondence between stability conditions and configuration of points in the plane will see that four points can form either a square or a triangle encircling a point, and these two configurations cannot be related by a braid. The point of the set $T$ is to choose one preferred configuration for each fundamental domain (later we will choose those of the convex kind).

The gradings might need to be considered a bit more closely. Indeed, we have two gradings in hand: the homological grading and the path length (quantum) grading. Their sum induces the grading of the $t$-structure underlying the choice of $\tau$. We will refer to this one as the triangulated grading, with shifts denoted by $[k]$. More precisely, we set
$X[k]\langle l \rangle =X\{k\}\langle l-k\rangle$. 

By considering the path length grading and the triangulated grading together, we induce a $\Z[q^{\pm 1},s^{\pm 1}]$ action on $K_0^\tau(\calC)$ by letting
\[
[X[1]]=s[X]\quad \mbox{and}\quad [X\langle 1\rangle]=q[X].
\]
A key consequence of the next result is that this module is free, which is somewhat counterintuitive if one is used to classical Grothendieck groups, where the following triangle
\[
X[1]\rightarrow (X\xrightarrow{\id} X[1]) \rightarrow X
\]
induces $[X[1]]=-[X]$. The reason why we do not have this relation is that the above triangle is not thin, see Remark~\ref{rem:idtriangle}. 
Moreover,  
\[
sq[X]= [X[1]\langle 1\rangle]=[X\{1\}\langle 0 \rangle]
\]
shows that the homological shift also induces a $\Z$-action on $K_0^\tau(\calC)$ that we will denote by \[[X\{1\}]=t[X],\]
where obviously $t$ is identified with $sq$.

In the sequel of the paper, we will work with this $\Z[q^{\pm 1},t^{\pm 1}]$-module structure and Theorem~\ref{thm:freeK0} specializes in our case as follows.
\begin{corollary} \label{cor:freeK0}
Let $\tau\in T$ and let $\Delta_{\tau}$ be the corresponding set of thin triangles. Then the extriangulated Grothendieck group can be described as follows
\[
K_0^\tau(\calC)\simeq \Z[q^{\pm 1},t^{\pm 1}]\langle \alpha, \; \alpha\in \Phi^{+}\rangle.
\]
Here $\Phi^+$ stands for the set of positive roots of the root system of Dynkin type $A$ and can be explicitly described in terms of the simple roots $\alpha_i$'s as
\[
\Phi^+=\{\alpha_{ij}=\alpha_i+\alpha_{i+1}+\cdots \alpha_{j-1},\;1\leq i < j\leq n+1 \} .
\]
\end{corollary}
\begin{proof}
This is a direct application of Theorem~\eqref{thm:freeK0} and of the fact that the stables are classified by the positive roots. %Up to the braid group action, t
This is \cite[Theorem 2.7]{LQ} (see also~\cite[Prop. 4.2 and Theorem 1.2]{BDL2} for a more general and constructive statement).
\end{proof}

\subsection{Recovering the LKB representation}

Following the strategy from Section~\ref{sec:linearrep}, we will reconstruct the LKB representation in two steps: we will get $q,t$-permutation matrices that should be thought of as some kind of holonomy matrices for a groupoid representation. Then we will pull everything back to a single basepoint, producing the LKB matrices.

Recall that $T$ represents the orbit of $\tau_0$ under the action of the braid group. It might happen that, for some braid $\beta$, the stability conditions $\beta\tau_0$ and $\tau_0$ share the same collection of stable objects and thus have the same combinatorial data to produce the extriangulated structure. We would like to identify such stability conditions. In order to do so, we define an equivalence relation on $T$ by saying that $\tau\sim \tau'$ if they share the same set of stable objects. This yields the set $\tilde{T}=T/\sim$. We will often identify a stability condition $\tau$ with its class in $\tilde{T}$.

For each $\tau\in \tilde{T}$, we make a specific choice of stables in the heart with prescribed $q$ -- and hence~$t$ --  gradings. Denote $\mathbf{B}_\tau$ this choice. This will also induce a preferred basis for the extriangulated Grothendieck group.

For any braid $\beta$ and any choice $\tau_1$, denote $\tau_2=\beta \tau_1$ the image of $\tau_1$ under the action of $\beta$ on the moduli space of stability conditions. Then $\beta \mathbf{B}_{\tau_1}$ is a set of stable objects that generates all stables of $\tau_2$ under grading shifts

\[
\forall S_i \in \mathbf{B}_{\tau_2}, \; \exists  k, l \in \Z  \mbox{ and } S'_j \in \mathbf{B}_{\tau_1} \mbox{ such that } \beta S'_j=S_i\{k\}\langle l\rangle.
\]
From there we create the matrix $P_{\tau_1}(\beta)$ with $t^kq^l$ in position $(i,j)$.

\begin{remark} \label{rem:prod}
We naturally have
    \[
P_{\tau_1}(\beta_2\beta_1)=P_{\beta_1\tau_1}(\beta_2)P_{\tau_1}(\beta_1).
    \]
\end{remark}

Now let us choose as basepoint $\tau_0^{\text{ref}}$ which is the stability condition $\tau_0$ together with preferred set $\mathbf{B}_{\tau_0}$ being the following collection
\begin{equation}
\label{eq:Pij}
P_{i,j}=P_i\{i-j+1\}\langle -i+j-1\rangle \rightarrow P_{i+1}\{i-j+2\}\langle-i+j-2\rangle\rightarrow \cdots \rightarrow P_{j-1}
\end{equation}
for all $1 \leq i < j \leq n+1$.

Recall that we defined an identification system to be a collection of invertible matrices expressing the following isomorphisms
    \[
    \forall \tau_1\in \tilde{T},\; M_{\tau_1}:K_0^{\tau_0^{\text{ref}}}(\calC)\rightarrow K_0^{\tau_1}(\calC).    
    \]
We highlight the fact that we will not choose the identity matrix for $M_{\tau_0}$, which amounts to another choice of basis for $K_0^{\tau_0}(\calC)$ than the one induced by $\mathbf{B}_{\tau_0}$.

Definition~\ref{def:compatible} rephrases into saying that an identification system is compatible with the action of the braid group if
    \[
   \forall \beta, \forall\tau_1,\tau_2\in \tilde{T}, M_{\beta\tau_1}^{-1}P_{\tau_1} (\beta)M_{\tau_1}=M_{\beta\tau_2}^{-1}P_{\tau_2} (\beta)M_{\tau_2}.
    \]
The above means that the two endomorphisms of $K_0^{\tau_0^{\text{ref}}}(\calC)$ obtained as compositions should be equal
\begin{equation} \label{eq:comp}
\begin{tikzpicture}[anchorbase,scale=.7,every node/.style={transform shape}]
\node [circle, draw] (ref) at (0,0) {$\tau_0^{\text{ref}}$};
\node [circle, draw] (T1) at (-3,2) {$\tau_1$};
\node [circle, draw] (BT1) at (-1,3.5) {$\beta\tau_1$};
\draw [->] (ref) -- node [midway,left] {\small $M_{\tau_1}$} (T1);
\draw [->] (T1) to  [out=50,in=-160] node[midway, above, xshift=-.3cm] {$P_{\tau_1}(\beta)$}  (BT1);
\draw [->] (BT1) -- node [midway,left] {\small $M_{\beta\tau_1}^{-1}$} (ref);
\node [circle, draw] (T2) at (1,3.5) {$\tau_2$};
\node [circle, draw] (BT2) at (3,2) {$\beta\tau_2$};
\draw [->] (ref) -- node [midway, left] {\small $M_{\tau_2}$} (T2);
\draw [->] (T2) to  [out=-20,in=130] node[midway, above, xshift=.3cm] {$P_{\tau_2}(\beta)$}  (BT2);
\draw [->] (BT2) -- node [midway,right] {\small $M_{\beta\tau_2}^{-1}$} (ref);
\end{tikzpicture}
\end{equation}

The following statement is merely an observation.
\begin{proposition}
    A compatible system yields a braid group representation on 
    $\Z[q^{\pm 1},t^{\pm 1}]\langle \alpha \in \Phi^{+}\rangle.$
\end{proposition}
\begin{proof}
The key point is the composition statement, which is a consequence of Remark~\ref{rem:prod}.
\end{proof}

Here is the main application of our setup.

\begin{theorem} \label{thm:KhSToLKB}
    There exists an identification system for $(B_{n+1},\tilde{T})$ that recovers the LKB representation.
\end{theorem}

Again, we emphasize the fact that, in comparison to Lemma~\ref{lem:triv}, we have imposed that the identification system should respect some symmetries of the category, which translates into reduced number of elements in $\tilde{T}$. The main goal in the proof of Theorem~\ref{thm:KhSToLKB} will be to show that these symmetries translate well to the LKB setting.

\begin{proof}

Let us fix some conventions. Recall that the matrix we get by looking at the decategorification of the categorical action of a braid $\beta$ from $K_0^{\tau}(\calC)$ to $K_0^{\beta\tau}(\calC)$ will be denoted $P_\tau(\beta)$. On the other hand, recall that $\rho$ is the LKB representation under the Paoluzzi-Paris conventions. Because of the conventions in the different references, we will identify the categorical action of $\sigma_i$ with $\rho(\sigma_i^{-1})$. Given $\beta$, we will denote $\tilde{\beta}$ the braid obtained by taking the inverse of the letters (note that we keep the order, so this is a group homomorphism, not the group anti-homomorphism given by the inverse map).

Assume given the matrix $M_{\tau_0}$ and consider a braid $\beta$ together with $\tau$ the representative of $\beta \tau_0$ in $\tilde{T}$.
\begin{equation} \label{eq:protosystem1}
\begin{tikzpicture}[anchorbase]
\node [circle, draw] (ref) at (0,0) {$\tau_0^{\text{ref}}$};
\node [circle, draw] (T0) at (0,4) {$\tau_0$};
\node [circle, draw] (T1) at (2.8,2.8) {$\tau$};
\draw [->] (ref) -- node [midway,left] {\small $M_{\tau_0}$} (T0);
\draw [->] (T0) to  [out=0,in=135] node[midway, above] {$P_{\tau_0}(\beta)$}  (T1);
\draw [->] (ref) to [out=70,in=135] (.8,.8) node [above] {$\rho(\tilde{\beta})$} to [out=-45,in=20] (ref);
\end{tikzpicture}
\end{equation}

Then there is a unique $M_{\tau}$ such that
\[
M_{\tau}^{-1}P_{\tau_0}(\beta)M_{\tau_0}=\rho(\tilde{\beta}).
\]

This process allows to define a matrix $M_\tau$ for any $\tau\in \tilde{T}$, except $\tau_0$ (where a choice was already made). However, we need to check that the definition does not depend on the choice of a particular braid $\beta$. Assume that $\beta \tau_0$ and $\beta'\tau_0$ have the same representative.

This means that the images of the stables of $\tau_0$ under $\mu=(\beta')^{-1}\beta$ are shifts of the stables. We will first prove that the shift is uniform on all stables. Notice that it suffices to show this for $P_i=P_{i,i+1}$, as the other stables are obtained by cone operations and shifts from those. Let us fix $i$, and define $S$, $S'$ two $\tau_0$-stables, and $k$ and $l$ in $\Z$ by $\mu(P_i)=S[k]$ and $\mu(P_{i+1})=S'[l]$. Then we use that $P_{i,i+2}=P_i\rightarrow P_{i+1}$, so that $\mu(P_{i,i+2})$ is the cone of a map between $S[k]$ and $S'[l]$. Knowing that the cone is contained in a shift of the heart implies that $k=l$. Doing this for all $i$'s implies that $\mu (\oplus_i P_i)$ is concentrated in just one $s$-degree. One can then apply~\cite[Theorem 4.1]{LQ} to show that $\mu=\gamma^l$ for some $l$. It should be noted that \cite{LQ} does not make mention of stability conditions. However, the baric structure that appears in~\cite{LQ} associated to $\gamma$ is closely related to $\tau_0$. In particular, the indecomposable objects in the heart of the baric structure are %just the same thing as the 
stables for $\tau_0$, and thus \cite[Theorem 4.1]{LQ} does apply to our situation.

To summarize, asking that $\beta'\tau_0=\beta\tau_0\in \tilde{T}$ amounts to saying that $(\beta')^{-1}\beta=\gamma^l$ for some value of $l$.

Let us denote $M_\beta$ and $M_{\beta'}$ the two matrices determined by $\beta$ and $\beta'$:
\[
M_{\beta}=P_{\tau_0}(\beta)M_{\tau_0}\rho(\tilde{\beta})^{-1}\mbox{ and }
M_{\beta'}=P(\beta')M_{\tau_0}\rho(\tilde{\beta'})^{-1}.
\]
So $M_{\beta}=M_{\beta'}$ if and only if $M_{\beta'}M_{\beta}^{-1}=\Id$, which expresses as 
\[
P_{\tau_0}(\beta')M_{\tau_0}\rho(\tilde{\beta'})^{-1}\rho(\tilde{\beta})M_{\tau_0}^{-1}P_{\tau_0}(\beta)^{-1}=\Id
\]
or equivalently
\[
M_{\tau_0}\rho(\tilde{\gamma}^l)M_{\tau_0}^{-1}=P_{\tau_0}(\gamma^l).
\]

So we will have a uniquely determined matrix if and only if
\begin{equation} \label{eq:condgamma}
M_{\tau_0}\rho(\tilde{\gamma})M_{\tau_0}^{-1}=P_{\tau_0}(\gamma).
\end{equation}

The core of the proof is to show that there exists a solution to this equation or, in other words, that $\rho(\tilde{\gamma})$ can be diagonalized with eigenvalues prescribed by $P_{\tau_0}(\gamma)$. We identify the preferred basis $\{[P_{i,j}]\}$ of $K_0^{\tau_0^{\text{ref}}}(\calC)$ with $\{e_{i,j}\}$ and, to simplify notation, from now on we simply write $P_{i,j}$ instead of $[P_{i,j}]$.

To simplify the computations, we will rather compute $\rho(\tilde{\gamma}^{-1})$, with $\tilde{\gamma}^{-1}=\sigma_n\cdots \sigma_1$ and to alleviate notations we forget $\rho$. Here are the formulas 
\begin{equation}
    \tilde{\gamma}^{-1}e_{i,j}=\begin{cases} \sum_{r=1}^{j-2}\sum_{s=r+1}^{n} (x-1)^2x^{s-r}y e_{r,s}-\sum_{r=1}^{j-2}(x-1)x^{n+1-r}ye_{r,n+1}\\
    \quad-x^{n-j+3}y e_{j-1,n+1}+\sum_{s=j}^nx^{s-j+2}y(x-1)e_{j-1,s}, \quad\text{if}\;i=1.\\
    xe_{i-1,j-1},\quad\text{if}\;i>1.
    \end{cases}
    \label{eq:gammaLKB}
\end{equation}

Let us prove these identities. 
If $i\neq 1$, then
\begin{gather*}
(\sigma_n\cdots \sigma_1)e_{i,j}=(\sigma_n\cdots \sigma_{i-1})e_{i,j}=(\sigma_n\cdots \sigma_i)(xe_{i-1,j}+(1-x)e_{i,j}) \\
 = (\sigma_n\cdots \sigma_{i+1})(xe_{i-1,j}+(1-x)e_{i+1,j})
 = \dots \\
 = (\sigma_n\cdots \sigma_{j-1})(xe_{i-1,j}+(1-x)e_{j-1,j}) =(\sigma_n\cdots \sigma_j)(xe_{i-1,j-1}) 
=xe_{i-1,j-1}.
\end{gather*}

Let us now focus on the case $i=1$. One easily computes
\[
(\sigma_n\cdots \sigma_1)e_{12}=-x^2y(\sigma_n\cdots \sigma_2)e_{12}=\dots %\\
=-x^{n+1}ye_{1,n+1}+\sum_{s=2}^nx^sy(x-1)e_{1,s}
\]
which is in line with \eqref{eq:gammaLKB}.

Finally, assume that $i=1$ and $j>2$. One computes
\begin{gather*}
    (\sigma_n\cdots \sigma_1)e_{1j}=(\sigma_n\cdots \sigma_2)(e_{2j}-xy(xy-1)e_{12})=\dots \\
    =(\sigma_n\cdots \sigma_{j-1})(\sum_{r=1}^{j-3}\sum_{s=r+1}^{j-2}(x-1)^2x^{s-r}ye_{r,s}
    -\sum_{r=1}^{j-3}(x-1)yx^{j-1-r}e_{r,j-1}-xy(x-1)e_{j-2,j-1}+e_{j-1,j}).
\end{gather*}
This equals
\begin{gather*}
(\sigma_n\cdots \sigma_j)(\sum_{r=1}^{j-2}\sum_{s=r+1}^{j-1}(x-1)^2x^{s-r}ye_{r,s}-\sum_{r=1}^{j-2}(x-1)x^{j-r}ye_{r,j}-x^2ye_{j-1,j}).
\end{gather*}
The first term is stabilized by $\sigma_n\cdots \sigma_j$. One can compute that
\begin{gather*}
    (\sigma_n\cdots \sigma_j)(-\sum_{r=1}^{j-2}(x-1)x^{j-r}ye_{rj})=
    (\sigma_n\cdots \sigma_{j+1})(-\sum_{r=1}^{j-2}(x-1)x^{j+1-r}ye_{r,j+1}+\sum_{r=1}^{j-2}(x-1)^2x^{j-r}ye_{r,j}) \\
    =\dots=\sum_{r=1}^{j-2}(x-1)x^{n+1-r}ye_{r,n+1}+\sum_{r=1}^{j-2}\sum_{s=j}^n(x-1)^2x^{s-r}ye_{r,s}.
\end{gather*}
For the last term, we get
\begin{gather*}
    (\sigma_n\cdots \sigma_j)e_{j-1,j}=(\sigma_n\cdots\sigma_{j+1})(xe_{j-1,j+1}+(1-x)e_{j-1,j})
    =x^{n-j+1}e_{j-1,n+1}+\sum_{s=j}^n(1-x)x^{s-j}e_{j-1,s}.
\end{gather*}

Collecting all terms yields~\eqref{eq:gammaLKB}.

Let us now explicitly define a solution for $M_{\tau_0}$ (and explicitly compute its inverse), as follows
\begin{align}
    \label{eq:M0}
    M_{\tau_0}e_{i,j}&=P_{i,j}+\sum_{s=i+1}^{j-1}(1-x^{-1})P_{i,s} \\
    M_{\tau_0}^{-1}P_{i,j}&=e_{i,j}-\sum_{s=i+1}^{j-1}x^{s-j}(x-1)e_{i,s}. \nonumber
\end{align}

It is then a painful computation to check that
\begin{equation} \label{eq:MgammainvMinv}
M_{\tau_0}\tilde{\gamma}^{-1}M_{\tau_0}^{-1}P_{i,j}=\begin{cases} xP_{i-1,j-1} \quad \text{if}\;i\neq 1, \\
(-x^{n-j+3}y)P_{j-1,n+1}\quad \text{if}\;i=1.
\end{cases}
\end{equation}
To see this, let us first consider the case $i>1$. Then
\begin{gather*}
P_{i,j}\xrightarrow{M_{\tau_0}^{-1}} e_{i,j}-\sum_{r=i+1}^{j-1} x^{r-j}(x-1)e_{i,r} \\
\xrightarrow{\sigma_n\cdots \sigma_1} x\left(e_{i-1,j-1}-\sum_{r=i+1}^{j-1}x^{r-j}(x-1)e_{i-1,r-1}\right)\xrightarrow{M_{\tau_0}}xP_{i-1,j-1}.
\end{gather*}

We will now consider the case $i=1$. We start from
\[
M_{\tau_0}^{-1}P_{1,j}=e_{1,j}-\sum_{r=2}^{j-1}x^{r-j}(x-1)e_{1,r}.
\]
Applying $\sigma_n\cdots \sigma_1$ to this yields (thanks to~\eqref{eq:gammaLKB}) the sum of the following 8 expressions.
\[
A:=\sum_{r=1}^{j-2}\sum_{s=r+1}^{n}(x-1)^2x^{s-r}ye_{r,s}
\]

\[
B:=-\sum_{r=1}^{j-2}(x-1)x^{n+1-r}ye_{r,n+1}
\]

\[
C:=-x^{n-j+3}ye_{j-1,n+1}
\]

\[
D:=\sum_{s=j}^nx^{s-j+2}y(x-1)e_{j-1,s}
\]

\begin{gather*}
E:=-\sum_{r=2}^{j-1}x^{r-j}(x-1)\sum_{s=1}^{r-2}\sum_{t=s+1}^{n}(x-1)^2x^{t-s}ye_{s,t} 
=\sum_{s=1}^{j-3}\sum_{t=s+1}^n(x-1)^2x^{t-j+2}y(1-x^{j-s-2})e_{s,t}
\end{gather*}

\begin{gather*}
F:=\sum_{r=2}^{j-1}x^{r-j}(x-1)\sum_{s=1}^{r-2}(x-1)x^{n+1-s}ye_{s,n+1} 
=-\sum_{s=1}^j-3 x^{n+3-j}(x-1)u(1-x^{j-s-2})e_{s,n+1}
\end{gather*}

\[
G:=\sum_{r=2}^{j-1}(x-1)x^{n+3-j}ye_{r-1,n+1}
\]

\begin{gather*}
H:=-\sum_{r=2}^{j-1}x^{r-j}(x-1)\sum_{s=r}^nx^{s-r+2}y(x-1)e_{r-1,s} 
=\sum_{s=1}^{j-1}\sum_{t=s+1}^nx^{t-j+2}y(x-1)^2e_{s,t}.
\end{gather*}

We now collect the coefficients of the $e_{r,s}$
\begin{itemize}
\item if $r\leq j-3$ and $s\leq n$, we have canceling contributions from $A$, $E$ and $H$;
\item if $r\leq j-3$ and $s=n+1$, we have canceling contributions from $B$, $F$ and $G$;
\item if $r=j-2$ and $s\leq n$, we have canceling contributions from $A$ and $H$;
\item if $r=j-2$ and $s=n+1$, we have canceling contributions from $B$ and $G$;
\item if $r=j-1$ and $s\leq n$, only $D$ contributes;
\item if $r=j-1$ and $s=n+1$, only $C$ contributes;
\item if $r\geq j$, we have no contribution.
\end{itemize}

So we are at
\[
(\sigma_n\cdots \sigma_1)M_{\tau_0}^{-1}P_{1,j}=(-x^{n-j+3}y)(e_{j-1,n+1}+\sum_{s=j}^nx^{s-n-1}(1-x)e_{j-1,s})
\]
from which we deduce
\[
M_{\tau_0}\tilde{\gamma}^{-1}M_{\tau_0}^{-1}P_{1,j}=-x^{n-j+3}P_{j-1,n+1}.
\]

We now want to compare the result from~\eqref{eq:MgammainvMinv} to the matrix obtained from the categorical action. Recalling the definition of $P_{i,j}$ from \eqref{eq:Pij}, one can explicitly compute
 \begin{equation}
P_{\tau_0}(\sigma_n^{-1}\cdots \sigma_1^{-1}) P_{i,j}=\begin{cases}
t^{2-j+n}q^{j-3-n}P_{j-1,n+1}\quad \text{if}\; i=1 ,\\
tq^{-1}P_{i-1,j-1} \quad\;\text{if}\;i\neq 1.
\end{cases} 
\label{eq:homgamma}
 \end{equation}
Since $\gamma$ takes the stables to the stables, and since the action decategorifies to the Weyl group action on roots, all that needs to be checked are the grading shifts.

Consider first the case $i\neq 1$. Then
\[
(\sigma_n^{-1}\cdots \sigma_1^{-1})P_{i,j}=(\sigma_n^{-1}\cdots \sigma_{i-1}^{-1})P_{i,j}.
\]
Applying $\sigma_{i-1}^{-1}$ to the object $P_{i,j}$ lets a copy of $P_{i-1}\langle -i+j-2\rangle \{i-j+2\}$ appear, which will be unaffected by the action of the next generators. Since, in the object $P_{i-1,j-1}$, the factor $P_{i-1}$ sits in homological degree $i-j+1$ and quantum degree $-i+j-1$, one reads a defect of $+1$ in homological grading and $-1$ in quantum grading, hence the formula \eqref{eq:homgamma}.

Now, for $i=1$, we have
\begin{gather*}
    (\sigma_n^{-1}\cdots \sigma_1^{-1})P_{1,j}=(\sigma_n^{-1}\cdots \sigma_{j-1})P_{j-1,j} 
    =(\sigma_n^{-1}\cdots \sigma_j^{-1})P_{j-1,j}\langle -2\rangle\{1\}.
\end{gather*}
Observe that $P_{j-1,j}$ is simply $P_{j-1}$. The term $P_{j-1}\langle-2\rangle\{1\}$ is stabilized by the remaining generators and it has to be compared with the copy of $P_{j-1}\langle n-j+1\rangle\{-n+j-1\}$ in $P_{j-1,n+1}$. This explains the coefficient $t^{2-j+n}q^{j-3-n}$ in~\eqref{eq:homgamma}.

The following identification of variables then yields the desired identities
\[
\begin{cases}
x=tq^{-1} \\
y=-t^{-1}
\end{cases}
\Leftrightarrow
\begin{cases}
q=-x^{-1}y^{-1} \\
t=-y^{-1}.
\end{cases}
\]

Now, we have proven that we have a well-defined identification system satisfying
\[
M_{\beta\tau_0}^{-1}P_{\tau_0}(\beta)M_{\tau_0}=\rho(\tilde{\beta}).
\]

To complete the proof, we need to check that it is compatible with the action of the braid group as formalized in Equation~\eqref{eq:comp}. As in Lemma~\ref{lem:triv}, we  see that for any stability condition $\tau$ and any braid $\beta$ mapping $\tau$ to $\tau'$, we have
\[
M_{\tau'}^{-1}P_{\tau}(\beta)M_{\tau}=\rho(\tilde{\beta}),
\]
which implies Equation~\eqref{eq:comp}. To prove the above formula, one can choose a braid $\beta_0$ mapping $\tau_0$ to $\tau$. We then have that
\[M_{\tau'}^{-1}P_{\tau_0}(\beta\beta_0)M_{\tau_0}=\rho(\tilde{\beta\beta_0}).\]
But
\[
P_{\tau_0}(\beta\beta_0)=P_{\tau}(\beta)P_{\tau_0}(\beta_0)
\mbox { and }
\rho(\tilde{\beta\beta_0})=\rho(\tilde{\beta})\rho(\tilde{\beta_0}).
\]
The conclusion follows after noticing that
\[
\rho(\tilde{\beta_0})=M_{\tau}^{-1}P_{\tau_0}(\beta_0)M_{\tau_0}.
\]
We thus have a compatible system that recovers the LKB representation.
\end{proof}

\subsection{Explicit matrices and another representation}

We will now give an explicit description of the bases ${\bf B}_{\tau}$ for the stability conditions involved in $\gamma$, as well as the relevant identification system.

Let us consider the following situation:
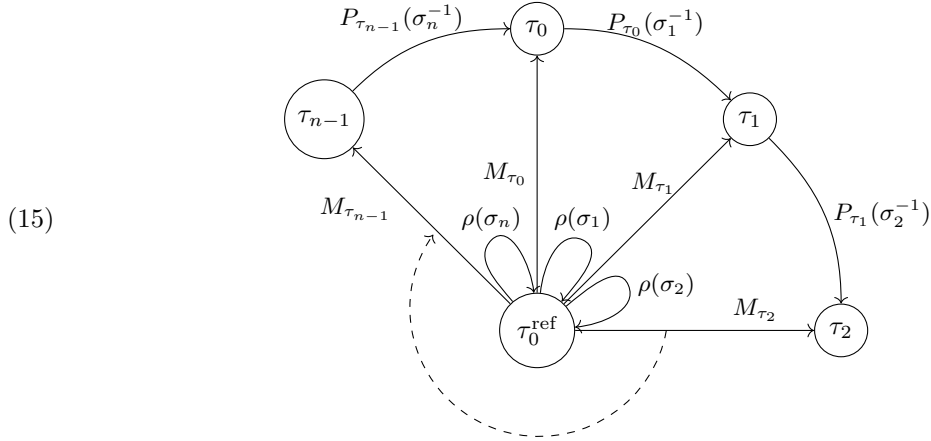
\begin{equation} \label{eq:protosystem2}
\begin{tikzpicture}[anchorbase]
\node [circle, draw] (ref) at (0,0) {$\tau_0^{\text{ref}}$};
\node [circle, draw] (T0) at (0,4) {$\tau_0$};
\node [circle, draw] (T1) at (2.8,2.8) {$\tau_1$};
\draw [->] (ref) -- node [midway,left] {\small $M_{\tau_0}$} (T0);
\draw [->] (T0) to  [out=0,in=135] node[midway, above] {\small $P_{\tau_0}(\sigma_1^{-1})$}  (T1);
\draw [->] (ref) to [out=85,in=155] (.6,1.2) node [above] {\small $\rho(\sigma_1)$} to [out=-25,in=50] (ref);
\draw [->] (ref) -- node [midway,left,xshift=.5cm,yshift=.5cm] {\small $M_{\tau_1}$} (T1);
\node [circle, draw] (T2) at (4,0) {$\tau_2$};
\draw [->] (ref) -- node [midway,above,xshift=.8cm] {\small $M_{\tau_2}$} (T2);
\draw [->] (T1) to [out=-45,in=90] node [midway,right] {\small $P_{\tau_1}(\sigma_2^{-1})$} (T2);
\draw [->] (ref) to [out=40,in=112] (1.2,.6) node [above,right] {\small $\rho(\sigma_2$)} to [out=-67,in=5] (ref);
\draw [dashed,->] (1.7,0) arc(-10:-215:1.7);
\node [circle,draw] (Tnm1) at (-2.8,2.8) {$\tau_{n-1}$};
\draw [->] (ref) --  node [midway,below,xshift=-1cm,yshift=.5cm] {\small $M_{\tau_{n-1}}$} (Tnm1);
\draw [->] (Tnm1) to [out=45,in=180] node [midway,above,xshift=-.3cm] {\small $P_{\tau_{n-1}}(\sigma_{n}^{-1})$} (T0);
\draw [->] (ref) to [out=130,in=-135] (-.6,1.2) node [above] {\small $\rho(\sigma_{n})$} to [out=45,in=95] (ref);
\end{tikzpicture}
\end{equation}

Recall from Section~\ref{sec:2varK0} that $\tau_0$ corresponds to the left-to-right orientation of the Dynkin diagram: differential maps involved in its stables are only of the kind $P_i\rightarrow P_{i+1}\{1\}\langle -1\rangle$, that follow the orientation of the Dynkin diagram. In \cite{LQ}, other orientations of the Dynkin diagram were considered, and the baric structures that are attached to them turn out to be realized by the stability conditions $\tau_k$. Understanding this depiction by orientations of the Dynkin diagram will prove handy to organize the description of the stables of the $\tau_k$'s.

Working inductively, one can easily show that  each of the $\tau_k$'s corresponds to stability conditions that come from the following orientation of the Dynkin diagram:
\[
\tau_k\quad \leftrightarrow \quad
\begin{tikzpicture}[anchorbase]
\node at (-0.5,0) {\textbullet};
\node at (-0.5,0.5) {\small $1$};
\node (dotsm) at (0,0) {$\cdots$};
\node (km1) at (1,0) {\textbullet};
\node at (1,.5) {\small $k-1$};
\node (k) at (2,0) {\textbullet};
\node at (2,.5) {\small $k$};
\node (kp1) at (3,0) {\textbullet};
\node at (3,0.5) {\small $k+1$};
\node (kp2) at (4,0) {\textbullet};
\node at (4,0.5) {\small $k+2$};
\node (dotsp) at (5,0) {\dots};
\node at (5.5,0) {\textbullet};
\node at (5.5,0.5) {\small $n$};
\draw [->] (dotsm) -- (km1);
\draw [->] (km1) -- (k);
\draw [->] (kp1) -- (k);
\draw [->] (kp1) -- (kp2);
\draw [->] (kp2) -- (dotsp);
\end{tikzpicture}
\]
The meaning of this statement is that the stables will be realized by reduced complexes using only differential maps that respect the above orientation. For example, applying $\sigma_1^{-1}$ on $\tau_0$--stables from~\eqref{eq:Pij} will affect only $P_{i,j}$ for $i=1$ or $2$. If $i=1$, then there is a piece $P_1\rightarrow P_2\{1\}\langle -1\rangle$ that is turned into $P_2\{1\}\langle -1\rangle$. If $i=2$, then $P_2$ produces under $\sigma_1^{-1}$ a copy of $P_2\rightarrow P_1\{1\}\langle -1\rangle$, which agrees with the new orientation of the edge between vertices $1$ and $2$ in $\tau_1$.

We make the following choices. For the basis ${\bf B}_{\tau_k}=\{P_{i,j}^k\}_{i,j}$, we choose the isoclasses in $K_0^{\tau_k}(\mathcal{C})$ of the complexes
\begin{equation}\label{Pijk}
P_{i,j}^k=
\begin{tikzpicture}[anchorbase]
\node (i) at (-2,0) {\small $P_i$};
\node (dotsm) at (-.5,0) {$\cdots$};
\node (km1) at (1,0) {\small $P_{k-1}$};
\node (k) at (2.5,0) {\small $P_k$};
\node (kp1) at (1,-1) {\small $P_{k+1}$};
\node at (1,-.5) {$\oplus$};
\node (kp2) at (2.5,-1) {\small $P_{k+2}$};
\node at (2.5,-.5) {$\oplus$};
\node (dotsp) at (4,-1) {\dots};
\node (jm1) at (5.5,-1) {\small $P_{j-1}$};
\draw [->] (i) -- (dotsm);
\draw [->] (dotsm) -- (km1);
\draw [->] (km1) -- (k);
\draw [->] (kp1) -- (k);
\draw [->] (kp1) -- (kp2);
\draw [->] (kp2) -- (dotsp);
\draw [->] (dotsp) -- (jm1);
\end{tikzpicture}
\end{equation}
with the rightmost term lying in homological and quantum grading $0$, unless $i\leq k$ and $j\geq k+3$. In the latter case, the rightmost term is $P_{j-1}\{-1\}\langle 1\rangle$. Shifts of the other terms follow.

\begin{lemma}\label{alphaornotalpha}
Let $M_{\tau_k}$ be the matrices defined by
\begin{equation} \label{eq:alpha_formula}
M_{\tau_k}e_{i,j}=\sum_{i',j'}\delta_{i',j'}^{i,j}\alpha^{c_{i',j'}^{i,j}}P^k_{i',j'},
\end{equation}
where 
\begin{itemize}
    \item  $\alpha=1-x^{-1}$;
    \item $\delta_{i',j'}^{i,j}=1$ if $P^k_{i',j'}$ is a quotient of $P^k_{i,j}$, and $\delta_{i',j'}^{i,j}=0$ otherwise;
    \item $c_{i',j'}^{i,j}$ is the number of irreducible components of $\mathrm{Cone}(P^k_{i',j'}\rightarrow P^k_{i,j})$.
\end{itemize}
 This yields the compatible identification system of Theorem~\ref{thm:KhSToLKB}.
\end{lemma}

\begin{proof}
Recall from the proof of Theorem~\ref{thm:KhSToLKB} that we made the following identification of variables: $x=tq^{-1}$ and $y=-t^{-1}$. 
In order to reduce the complexity of the computations, we check that, for any~$k$, 
\[
P_{\tau_0}(\sigma_k^{-1}\cdots \sigma_1^{-1})M_{\tau_0}=M_{\tau_k}\rho(\sigma_k\cdots \sigma_1).
\]

Let us make Equation~\eqref{eq:alpha_formula} more explicit. We have
\begin{equation}
    M_{\tau_k}e_{i,j}=P_{i,j}^k+\sum_{\substack{j'=i+1 \\j'\neq k+1}}^{j-1}\alpha P_{i,j'}^k + \delta_{k+1>i}\sum_{j'=k+2}^{j-1} \alpha^2P^k_{k+1,j'}+\delta_{i<k+1}\delta_{j>k+1}\alpha P^k_{k+1,j}
\end{equation}
where in the above sums the term where the upper script is smaller than the lower script are considered as zero.

Recall that for $k=0$, $M_{\tau_0}$ has the nicer expression
\[
M_{\tau_0}e_{i,j}=P_{i,j}^0+\sum_{j'=i+1}^{j-1}\alpha P_{i,j'}^0.
\]

Furthermore, a direct computation shows that
\begin{equation} \label{eq:action_partial_gamma}
\sigma_k^{-1}\cdots \sigma_1^{-1}P_{i,j}^{0}=
\left\{
 \begin{alignedat}{2}
& P_{k+1,j}^k \quad && \text{if }i=1,\; k<j-1\\
& P_{k,k+1}^k\{1\}\langle -2\rangle \quad && \text{if } i=1,\; k=j-1\\
& P_{j-1,k+1}^k\{k-j+2\}\langle j-k-3\rangle \quad && \text{if }i=1,\; k\geq j \\
& P_{i-1,j-1}^k\{1\}\langle-1\rangle \quad && \text{if }1<i\leq k \text{ and }j\leq k+1 \\
& P_{i-1,j}^k\{1\}\langle-1\rangle \quad && \text{if }1<i\leq k\text{ and }j\geq k+2 \\
& P_{i-1,j}^k\{1\}\langle-1\rangle \quad && \text{if } i=k+1 \\
&P_{i,j}^k \quad && \text{if }i>k+1.
\end{alignedat}
\right.
\end{equation}

From the computations in the proof of Theorem~\ref{thm:KhSToLKB}, we have 
\begin{equation}
\sigma_k\cdots \sigma_1e_{i,j}=\left\{
\begin{alignedat}{2}
& \sum_{r=1}^{k-1}\sum_{s=r+1}^k (x-1)^2x^{s-r}y e_{r,s}-\sum_{r=1}^{k-1}(x-1)yx^{k+1-r}e_{r,k+1} && \\
&\qquad -xy(x-1)e_{k,k+1}+e_{k+1,j} && \text{if }i=1,\;j-1>k \\
& \sum_{r=1}^{j-2}\sum_{s=r+1}^{k}(x-1)^2x^{s-r}ye_{r,s}-\sum_{r=1}^{j-2}(x-1)x^{k+1-r}ye_{r,k+1} && \\
& \qquad -x^{k-j+3}ye_{j-1,k+1}+\sum_{s=j}^{k}x^{s-j+2}y(x-1) e_{j-1,s}&& \text{if }i=1,\; j\leq k+1 \\
& xe_{i-1,j}+(1-x)e_{k+1,j} && \text{if }i>1,\; k<j-1 \\
& xe_{i-1,j-1} && \text{if }k+1\geq i>1,\; k\geq j-1 \\
& e_{i,j} && \text{if } i>k+1.
\end{alignedat}
\right.
\end{equation}

The end of the proof is based on the following disjunction of cases:
\begin{itemize}
\item $i>1$, $j\leq k+1$;
\item $i>1$, $j\geq k+2$;
\item $i=1$, $k<j-1$;
\item $i=1$, $j=k+1$;
\item $i=1$, $j\leq k$.
\end{itemize}
We explicitly treat the case $i=1,\;k<j-1$, and leave the other ones to the reader.

Let us first compute
\begin{align}
e_{1,j}&\xrightarrow{M_{\tau_0}} P_{1,j}^0+\sum_{j'=2}^{j-1}\alpha P^0_{1,j'} \nonumber \\
&\xrightarrow{P_{\tau_0}(\sigma_k^{-1}\cdots \sigma_1^{-1})} P_{k+1,j}^k+\alpha \sum_{j'=2}^{k}(-x^{k-j+3}y)P^k_{j'-1,k+1}+\alpha(-x^2y)P_{k,k+1}^k+\sum_{j'=k+2}^{j-1}\alpha P_{k+1,j'}^{k} .\label{eq:3part1}
\end{align}

We compare this expression to
\begin{align}
e_{1,j}&\xrightarrow{\rho(\sigma_k\cdots \sigma_1)}\sum_{r=1}^{k-1}\sum_{s=r+1}^k (x-1)^2x^{s-r}ye_{r,s}-\sum_{r=1}^{k-1}(x-1)yx^{k+1-r}e_{r,k+1}-xy(x-1)e_{k,k+1}+e_{k+1,j} \nonumber \\
&\xrightarrow{M_{\tau_k}}\sum_{r=1}^{k-1}\sum_{s=r+1}^k
(x-1)^2x^{s-r}y(P_{r,s}^k+\sum_{j'=r+1}^{s-1}\alpha P^k_{r,j'}) \nonumber \\
&\quad -\sum_{r=1}^{k-1}(x-1)yx^{k+1-r}(P^k_{r,k+1}+\sum_{j'=r+1}^k\alpha P^k_{r,j'}) \nonumber \\
&\quad -xy(x-1)P_{k,k+1}^k+P^k_{k+1,j}+\sum_{j'=k+2}^{j-1}\alpha P_{k+1,j'}^k. \label{eq:3part2}
\end{align}

We denote by $A$ through $G$ the terms that appear in \eqref{eq:3part2} and that express as follows (once one rearranges the summations in the terms $B$):
\[
A=\sum_{r=1}^{k-1}\sum_{s=r+1}^k(x-1)^2x^{s-r}yP_{r,s}^k
\]
\[
B=\sum_{r=1}^{k-1}\sum_{j'=r+1}^{k-1}\alpha x^{j'+1-r}(1-x^{k-j'})(1-x)yP^k_{r,j'}
\]
\[
C=-\sum_{r=1}^{k-1}(x-1)yx^{k+1-r}P_{r,k+1}^k
\]
\[
D=-\sum_{r=1}^{k-1}\sum_{j'=r+1}^{k}\alpha (x-1)yx^{k+1-r}P_{r,j'}^k
\]
\[
E=-xy(x-1)P_{k,k+1}^k
\]
\[
F=P^k_{k+1,j}
\]
\[
G=\sum_{j'=k+2}^{j-1}\alpha P^k_{k+1,j'}.
\]

Now, let us compare the terms in Equations~\eqref{eq:3part1} and \eqref{eq:3part2}, recalling that $\alpha=1-x^{-1}$,

 \begin{itemize}
\item $P^k_{k+1,j}$ has coefficient $1$ in \eqref{eq:3part1} and appears only in the term $F$ in \eqref{eq:3part2} with the same coefficient;
\item $P^k_{r,k+1}$ with $1\leq r\leq k-1$ appears in \eqref{eq:3part1} with coefficient $-\alpha x^{k-r+2}y$ and appears only in term $C$ in\eqref{eq:3part2} with the same coefficient;
\item $P_{k,k+1}^k$ has coefficient $-\alpha x^2y$ in \eqref{eq:3part1} and appears only in the term $E$ in \eqref{eq:3part2} with the same coefficient;
\item $P^k_{k+1,s}$ with $s<j$ appears in \eqref{eq:3part1} with coefficient $\alpha$ and appears only in term $G$ in \eqref{eq:3part2} with the same coefficient.
 \end{itemize}

This exhausts all the terms that appear in \eqref{eq:3part1}. We are then left to check that all the remaining terms of \eqref{eq:3part2} cancel. Only $P_{r,s}^k$ with $r\leq k-1$, $s\leq k$ is in this case, and it appears in the terms $A$, $B$ and $C$. One easily checks that the sum of the coefficients is indeed zero.
 
 \end{proof}

\begin{remark}~\label{rem:matrices}
The expression of the matrices as in \eqref{eq:alpha_formula} is rather intriguing. It would be interesting to know how much this formula can be generalized to all stability conditions and to better understand the coefficients of the count in terms of dimensions of morphism spaces. Explicit computations show that the formulas do not directly hold in general. %\HQ{give an explicit example?}
\end{remark}

It appears that the formulas of Lemma~\ref{alphaornotalpha} yield a representation of the braid group only for two values of $\alpha$ -- either the one of Lemma~\ref{alphaornotalpha}, or $\alpha=0$, as stated below. This version is reminiscent of the Tong-Yang-Ma representation~\cite{TongYangMa}.

\begin{proposition}
    Replacement of the symbol $\alpha$ by $0$ in formula \eqref{eq:alpha_formula} of Lemma~\ref{alphaornotalpha}  produces a representation of the braid group by generalized permutation matrices. 
\end{proposition}

\begin{proof}
In this version, we are in a situation similar to \eqref{eq:protosystem2}. We can get rid of the step consisting in taking inverses of the generators, coming from incoherent conventions between LKB representation and categorical action.
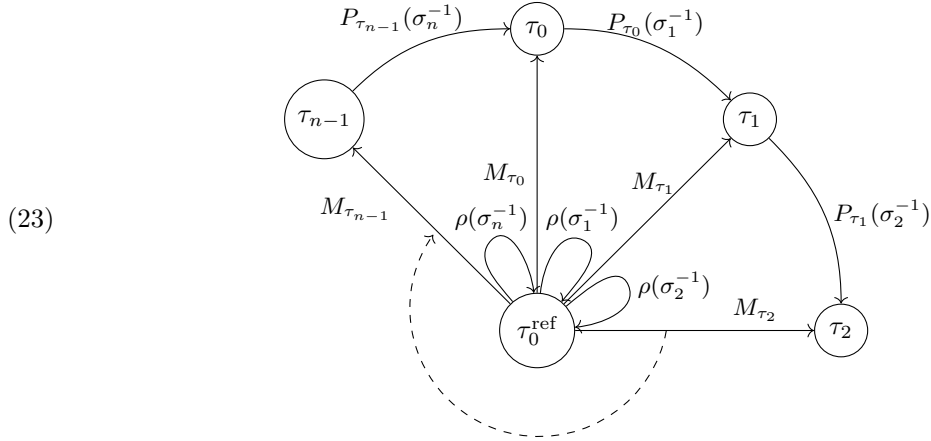
\begin{equation} \label{eq:protosystem3}
\begin{tikzpicture}[anchorbase]
\node [circle, draw] (ref) at (0,0) {$\tau_0^{\text{ref}}$};
\node [circle, draw] (T0) at (0,4) {$\tau_0$};
\node [circle, draw] (T1) at (2.8,2.8) {$\tau_1$};
\draw [->] (ref) -- node [midway,left] {\small $M_{\tau_0}$} (T0);
\draw [->] (T0) to  [out=0,in=135] node[midway, above] {\small $P_{\tau_0}(\sigma_1^{-1})$}  (T1);
\draw [->] (ref) to [out=85,in=155] (.6,1.2) node [above] {\small $\rho(\sigma_1^{-1})$} to [out=-25,in=50] (ref);
\draw [->] (ref) -- node [midway,left,xshift=.5cm,yshift=.5cm] {\small $M_{\tau_1}$} (T1);
\node [circle, draw] (T2) at (4,0) {$\tau_2$};
\draw [->] (ref) -- node [midway,above,xshift=.8cm] {\small $M_{\tau_2}$} (T2);
\draw [->] (T1) to [out=-45,in=90] node [midway,right] {\small $P_{\tau_1}(\sigma_2^{-1})$} (T2);
\draw [->] (ref) to [out=40,in=112] (1.2,.6) node [above,right] {\small $\rho(\sigma_2^{-1})$} to [out=-67,in=5] (ref);
\draw [dashed,->] (1.7,0) arc(-10:-215:1.7);
\node [circle,draw] (Tnm1) at (-2.8,2.8) {$\tau_{n-1}$};
\draw [->] (ref) --  node [midway,below,xshift=-1cm,yshift=.5cm] {\small $M_{\tau_{n-1}}$} (Tnm1);
\draw [->] (Tnm1) to [out=45,in=180] node [midway,above,xshift=-.3cm] {\small $P_{\tau_{n-1}}(\sigma_{n}^{-1})$} (T0);
\draw [->] (ref) to [out=130,in=-135] (-.6,1.2) node [above] {\small $\rho(\sigma_{n}^{-1})$} to [out=45,in=95] (ref);
\end{tikzpicture}
\end{equation}
The matrices $M_{\tau_k}$ are just identity matrices, identifying basis elements with the roots they categorify.
Hence we now drop the subscript $k$ and just write $P_{i,j}$. This choice forces the matrices associated to the generators $\sigma_i$ (which at the moment have no reason to braid). In turn, this forces the matrix associated to any braid $\beta$, and because we also want to impose that $M_{\beta\tau}$ is a permutation matrix as in the statement, this forces the choice of the basis $\mathbf{B}_{\beta\tau}$.
Because different braids mapping $\tau_0$ to the same $\tau$ might force different basis choices, we need to check that there exists a solution.

Recall from \eqref{Pijk} the explicit expression of the basis $\{P_{i,j}\}$. By applying the generators as in \eqref{eq:protosystem2}, one can explicitly compute the matrices
\begin{equation}
\sigma_{k+1}^{-1}P_{i,j}    =\left\{
\begin{alignedat}{2}
& tq^{-2}P_{k+1,k+2} \quad && \mathrm{if}\; i=k+1,\;j=k+2 \\
& P_{k+2,j} \quad && \mathrm{if}\; i=k+1,\;j>k+2 \\
& tq^{-1} P_{k+1,j} \quad && \mathrm{if}\; i=k+2 \\
& tq^{-1} P_{i,k+2} \quad && \mathrm{if} \; j=k+1 \\
& P_{i,k+1} \quad && \mathrm{if} \; i<k+1,\; j=k+2 \\
& P_{i,j} \quad && \mathrm{otherwise.}
\end{alignedat}
\right.
\end{equation}

Notice that $\sigma_k$ affects non-trivially basis vectors $P_{i,j}$ only if $i=k$ or $k+1$ or $j=k$ or $k+1$. In these cases, all it can do is exchange $k$ and $k+1$. So the commutation relation $\sigma_r\sigma_s=\sigma_s\sigma_r$ for $|r-s|>1$ is clear.

For the braid relation, consider the two words $\sigma_{k}^{-1}\sigma_{k+1}^{-1}\sigma_k^{-1}$ and $\sigma_{k+1}^{-1}\sigma_k^{-1}\sigma_{k+1}^{-1}$. Each of them will only affect elements $P_{i,j}$ that have $i$ or $j$ in $\{k,k+1,k+2\}$. Let us consider all cases
\begin{itemize}
    \item if $i=k$, $j=k+1$: one gets $t^2q^{-3}P_{k+1,k+2}$ in both cases;
    \item if $i=k$, $j=k+2$: one gets $t^2q^{-2}P_{k,k+2}$ in both cases;
    \item if $i=k$, $j>k+2$: one gets $P_{k+2,j}$ in both cases;
    \item if $i=k+1$, $j=k+2$: one gets $t^2q^{-3}P_{k,k+1}$ in both cases;
    \item if $i=k+1$, $j>k+2$: one gets $tq^{-1}P_{k+1,j}$ in both cases;
    \item if $i=k+2$: one gets $t^2q^{-2}P_{k,j}$ in both cases;
    \item if $j=k$: one gets $t^2q^{-2}P_{i,k+2}$ in both cases;
    \item if $j=k+1$, $i<k$: one gets $tq^{-1}P_{i,k+1}$ in both cases;
    \item if $j=k+2$, $i<k$: one gets $P_{i,k}$ in both cases. 
\end{itemize}
This exhausts the cases to check.

Now, we consider the question of the uniqueness of the choice of the basis. Let $\beta_1$ and $\beta_2$ be so that $\beta_1\tau_0$ and $\beta_2\tau_0$ have the same representative in $\tilde{T}$. This means that $\beta_1=\beta_2\gamma^l$ for some $l$.

Let us fix an arbitrary basis $\mathbf{b}$ for $\tau$, and consider $C_1$ (resp. $C_2$) the base change matrix between the chosen one and the one imposed by $\beta_1$ (resp. $\beta_2$). One reads
\[C_1P_{\tau_0}(\beta)=\rho(\beta_1),\;C_2P_{\tau}(\beta_2)=\rho(\beta_2).
\]
Above, the matrices $P$ are computed with respect to the chosen basis $\mathbf{b}$. We want
\[C_1=C_2\Leftrightarrow \rho(\beta_1)P_{\tau_0}(\beta_1)^{-1}P_{\tau_0}(\beta_2)\rho(\beta_2)^{-1}=\mathrm{Id}.
\]
This is equivalent to
\[
P_{\tau_0}(\gamma^k)=\rho(\gamma^k),
\]
which follows from our definitions.

\end{proof}

% Biblio
%% \bibliographystyle{amsalpha}
%% \bibliography{biblio}

  \providecommand{\bysame}{\leavevmode\hbox to3em{\hrulefill}\thinspace}
\providecommand{\MR}{\relax\ifhmode\unskip\space\fi MR }
\providecommand{\MRhref}[2]{%
  \href{http://www.ams.org/mathscinet-getitem?mr=#1}{#2}
}
\providecommand{\href}[2]{#2}

\end{document}